\documentclass{amsart}
\usepackage{amssymb,mathrsfs,MnSymbol}%,skak}
\usepackage{color}
\usepackage{hyperref}
\hypersetup{pdfborder=0 0 0}
\usepackage{hyperref}
\hypersetup{pdfborder=0 0 0}

\newcommand{\nn}{\nonumber}

\renewcommand{\tilde}{\widetilde}

\def\bC {\mathbf{C}}

\def\bN {\mathbf{N}}

\def\bR {\mathbf{R}}

\def\fH {\mathfrak{H}}

\def\cA {\mathcal{A}}

\def\cC {\mathcal{C}}
\def\cD {\mathcal{D}}
\def\cE {\mathcal{E}}

\def\cH {\mathcal{H}}

\def\cK {\mathcal{K}}
\def\cL {\mathcal{L}}

\def\cN {\mathcal{N}}
\def\cP {\mathcal{P}}
\def\cQ {\mathcal{Q}}
\def\cR {\mathcal{R}}
\def\cS {\mathcal{S}}

\def\cV {\mathcal{V}}
\def\cW {\mathcal{W}}

\def\scrH{\mathscr{H}}

\def\a {{\alpha}}
\def\b {{\beta}}

\def\Ga {{\Gamma}}
\def\de {{\delta}}
\def\eps {{\epsilon}}

\def\l {{\lambda}}
\def\L {{\Lambda}}
\def\si {{\sigma}}

\def\om {{\omega}}
\def\Om {{\Omega}}

\def\d {{\partial}}
\def\grad {{\nabla}}
\def\Dlt {{\Delta}}

\def\rstr {{\big |}}
\def\indc {{\bf 1}}

\def\la {\langle}
\def\ra {\rangle}
\def \La {\bigg\langle}
\def \Ra {\bigg\rangle}
\def \lA {\big\langle \! \! \big\langle}
\def \rA {\big\rangle \! \! \big\rangle}

\newcommand{\Dom}{\operatorname{Dom}}
\newcommand{\fDom}{\operatorname{Form-Dom}}

\newcommand{\Span}{\operatorname{span}}

\newcommand{\Tr}{\operatorname{trace}}

\newcommand{\Ker}{\operatorname{Ker}}

\newcommand{\Img}{\operatorname{Ran}}

\newcommand{\MKd}{\operatorname{dist_{MK,2}}}
\newcommand{\Op}{\operatorname{OP}}

\def\hb {{\hbar}}

\newcommand{\ba}{\begin{aligned}}
\newcommand{\ea}{\end{aligned}}

\newcommand{\be}{\begin{equation}}
\newcommand{\ee}{\end{equation}}

\newcommand{\lb}{\label}

\newtheorem{Thm}{Theorem}[section]

\newtheorem{Cor}[Thm]{Corollary}
\newtheorem{Lem}[Thm]{Lemma}
\newtheorem{Def}[Thm]{Definition}

\begin{document}

\title[Quantum and Semiquantum Pseudometrics]{Quantum and Semiquantum Pseudometrics and applications}

\author[F. Golse]{Fran\c cois Golse}
\address[F.G.]{CMLS, \'Ecole polytechnique, CNRS, Universit\'e Paris-Saclay , 91128 Palaiseau Cedex, France}
\email{francois.golse@polytechnique.edu}

\author[T. Paul]{Thierry Paul}
\address[T.P.]{CNRS \& LJLL Sorbonne Universit\'e 4 place Jussieu 75005 Paris, France}
\email{thierry.paul@upmc.fr}

\begin{abstract}
We establish a Kantorovich duality for he pseudometric $\cE_\hb$ introduced in [F. Golse, T. Paul, Arch. Rational Mech. Anal. \textbf{223} (2017), 57--94], obtained from the usual Monge-Kantorovich distance $\MKd$ between classical densities  by quantization of one side  of the two densities  involved. We show several type of inequalities comparing $\MKd$, $\cE_\hb$ and $MK_\hb$, a full quantum analogue of $\MKd$ introduced in [F. Golse,  C. Mouhot, T. Paul, 
Commun. Math. Phys. \textbf{343} (2016), 165--205], including an up to $\hbar$  triangle inequality for $MK_\hb$. Finally, we show that, when nice  optimal Kantorovich potentials exist for $\cE_\hb$, optimal couplings  induce classical/quantum optimal transports and  the  potentials are linked by a semiquantum Legendre
%-Lax-Hopf 
type transform.
\end{abstract}

\date{\today}

\maketitle

\tableofcontents

\section{Introduction and statement of some main results}\label{intro}

The Monge-Kantorovich distance, also called
Wasserstein distance, of exponent two on the phase-space $T^*\bR^{d}\sim \bR^{2d}$ is defined, for two probability measures 
%$\mu,\nu\in\cP(\bR^{2d})$ 
by
\be\label{defmkd}
\MKd(\mu,\nu)^2
=
\inf_{\pi\in\pi[\mu.\nu]}\int_{\bR^{2d}\times\bR^{2d}}((q-q')^2+(q-p')^2)\pi(dqdp,dq'dp')
\ee
where $\pi[\mu,\nu]$ is the set of couplings $\pi$ of $\mu,\nu$, i.e. the set of probability measures $\pi$ on $\bR^{2d}\times\bR^{2d}$ such that for all test functions $a,b\in C_c(\bR^{2d})$ we have that
%$\ \int_{\bR^{2d}\times\bR^{2d}}a(q,p)\pi(dqdp,dq'dp')=\int_{\bR^{2d}}a(q,p)\mu{dqdp}$ and $ 
%\int_{\bR^{2d}\times\bR^{2d}}b(q',p')\pi(dqdp,dq'dp')=\int_{\bR^{2d}}b(q',p')\nu{dq'dp'}$.
$$
\int_{\bR^{2d}\times\bR^{2d}}(a(q,p)\!+\!b(q',p'))\pi(dqdp,dq'dp')\!=\!\!\!\int_{\bR^{2d}}(a(q,p)\mu(dqdp)\!+\!b(q',p')\nu(dq'dp')).$$

Among the many properties of $\MKd$, let us mention the Kantorovich duality wich stipulates that
\be\label{kantowass}
\MKd(\mu,\nu)^2
=
\max_{\substack{a,b\in C_b(\bR^{d})\\
a(q,p)+b(q',p')\leq (q-q')^2+(p-p')^2}}
\int_{\bR^{2d}}(a(q,p)\mu(dqdp)+b(q,p)\nu(dqdp),
\ee
and the Knott-Smith-Brenier Theorem which says that, under certain conditions on $\mu,\nu$,  any coupling $\pi_{op}$ satisfying
\be
\MKd(\mu,\nu)^2
=
\int_{\bR^{2d}\times\bR^{2d}}((q-q')^2+(q-p')^2)\pi_{op}(dqdp,dq'dp')
\ee
is supported in the graph of the convex function  $\tfrac12(q^2+p^2)-a_{op}(q,p)$ where $a_{op}$ is an optimal function such that $a_{op}, b_{op}$ provide the max in \eqref{kantowass} for some $b_{op}$.

\noindent 
Finally,  $\tfrac12(q^2+p^2)-a_{op}(q,p)$ and  $\tfrac12(q^2+p^2)-b_{op}(q,p)$ are proven to be the Legendre transform  of each other.
\vskip 1cm
A quantum version of $\MKd$ was proposed in \cite{FGMouPaul} following the general rules of quantization consisting in replacing
\begin{itemize}
\item probability measures $\mu.\nu$ on phase-space $T^*\bR^{d}$  by quantum states $R,S$, i.e. density operators, i.e. positive trace one operators on $L^2(\bR^d)$
\item $\int_{T^*\bR^{d}}$ by $\Tr_{L^2(\bR^d)}$
\item couplings of $\mu,\nu$ 
by density operators $\Pi$ on 
$L^2(\bR^d)
\otimes L^2(\bR^d)
$ 
such that, for any bounded operators $A,B$, $\Tr_{L^2(\bR^d)\otimes L^2(\bR^d)}{(A\otimes I)\Pi)}=\Tr_{L^2(\bR^d)}{AS}$ and $\ \Tr_{L^2(\bR^d)\otimes L^2(\bR^d)}{(I\otimes B)\Pi)}=\Tr_{L^2(\bR^d)}{BR}$.
% for any bounded operators $A,B$ on $L^2(\bR^d)$.
\item the cost function $(q-q')^2+(q-p')^2$ by its Weyl pseudodifferential quantization 
$C=(x-x')^2+
(-i\hbar\nabla_x+i\hbar\nabla_{x'})^2
$ on 
$L^2(\bR^d\times\bR^d)$.
\end{itemize}
These considerations lead to the definition, for two density operators $R,S$ on $L^2(\bR^2)$,
\be\label{defmkhbar}
MK_\hb(R,S)^2=\inf_{\Pi\mbox{ coupling }R\mbox{ and }S}\Tr_{L^2(\bR^d)\otimes L^2(\bR^d)}{C\Pi}.
\ee

The pseudometric $MK_\hb$ has been extensively studied in \cite{FGMouPaul}, with applications to the study of the quantum mean-field limit uniformly in $\hbar$, used in \cite{ECFGTPaul} for quantum optimal transport considerations and applied in \cite{ECFGTPaul2} for the quantum bipartite matching  problem. In particular, a Kantorovich duality was proven for $MK_\hb$ in \cite{ECFGTPaul} expressed as the following identity
\be\label{dulamkhbar}
MK_\hb(R,S)^2=\sup_{\substack{A=A^*,B=B^*\in\cL(L^2(\bR^d))\\ A\otimes I+I\otimes B\leq C}}\Tr{(AR+BS)}
\ee
and the supremum was proven to be attended  for two oparors $\bar A,\bar B$ defined respectively on two Gelfand triplest surrounding $L^2(\bR^d)$ (see \cite{ECFGTPaul})

Though $MK_\hb$ is symmetric in its argument, it is not a distance as one can easily show (\cite{FGMouPaul}) that $MK^2_\hb\geq 2d\hbar$. Nevertheless, one of the main result of this article will be to prove the following (approximate) triangle inequality, valid for density operators $R,S,T$ (see Theorem \ref{T-IneqT2} $(iii)$ below)
\be\label{tri}
MK_\hb(R,T)
\leq MK_\hb(R,S)+
MK_\hb(S,T)+d\hb.
\ee

Actually, \eqref{tri} is proved by using a kind of ``semiquantum" generalisation of $\MKd$, defined in \cite{FGPaul} and constructed by, roughly speaking, applying the quantization rule aforementioned to only one on the two parts involved in $\MKd(\mu,\nu)$: 

for $f$ probability density on $\bR^{2d}$ and $R$ density operator on $L^2(\bR^d)$ we define
\be\label{defarma}
\cE_\hb(f,R)^2=\sup_{\substack{\Pi(q,p)\\
 density\ operators\\ such\ that\\ 
%\  on\ $L^2(\bR^d)$\ for\ a.e.\ (q,p)\\
\Tr\Pi(q,p)=f(q,p)\\and \ \int_{\bR^{2d}}\Pi(q,p)=R
}}
\int_{\bR^{2d}}\Tr_{L^2(\bR^d,dx)}{((q-x)^2+(p+i\hbar\nabla_x)^2)\Pi(q,p)}dqdp.
\ee
The pseudometric $\cE_\hb$ has been used in \cite{FGPaul} in order to derive  several results concerning the quantum, uniform in $\hbar$, mean-field derivation and in \cite{FGPaul,FGPaul2} for semicalssical propagation estimates involving low regularity of the potential and the initial data (in particular with respect to the dimension, i.e.  also to the number of particles present in the quantum evolution).

In the present paper, we  prove a Kantorovich duality for $\cE_\hb$ (Section \ref{kantodual}, Theorem \ref{T-DualityARMA}), namely
\be\label{dualce}
\cE_\hb(f,R)^2
=\sup_{\substack{a\in C_b(\bR^{2d}),\ B\in\cL(L^2(\bR^d))\\
a(q,p)+B\leq (q-x)^2+(p+i\hbar\nabla_x)^2}}
\int_{\bR^{2d}}a(q,p)f(q,p)dqdp+\Tr_{L^2(\bR^d)}BR,
\ee
 and then apply this duality to derive inequalities, such as \eqref{tri}, involving $MK_\hb,\ \cE_\hb$ and $\MKd$, Theorems \ref{T-LBdMKE} and \ref{T-IneqT2}.

In the last section of the paper, Section \ref{app3}, we investigate the semiquantum analogue of the Knott-Smith-Brenier Theorem and a semiquantum analogue of the Legendre transform: if
$$\cE_\hb(f,R)^2=
%\sup_{\substack{a\in C_b(\bR^{2d}),\ B\in\cL(L^2(\bR^d))\\
%a(q,p)+B\leq (q-x)^2+(p+i\hbar\nabla_x)^2}}
\int_{\bR^{2d}}a_{op}(q,p)f(q,p)dqdp+\Tr_{L^2(\bR^d)}B_{op}R,
$$ then $a(q,p):=\tfrac12(p^2+q^2-a_{op}(q,p))$  is the semiquantum-Legendre transform of  $B:=\tfrac12(-\nabla_x^2+x^2-B_{op})$, in the sense that
$$
a(q,p)=\sup_{\phi\in 
%L^2(\mathbb R^d)
\Dom(B)}(q\cdot\langle\phi| x|\phi\rangle+p\cdot\langle\phi|-i\hbar\nabla_x|\phi\rangle-
\langle\phi|B|\phi\rangle).
$$
%%%%%%%%%%%%%%%%%%%%%%%%%%%%%%%%%%%%%%%%%%%%%%%%%%%%%%%%%%%%%%%%%%%%%%%%%%%%%%%%%%%%%%%%%%%%%%%%%%%%%%%%%%%%%%%%%%%%%%
\section{Preliminaries}
%%%%%%%%%%%%%%%%%%%%%%%%%%%%%%%%%%%%%%%%%%%%%%%%%%%%%%%%%%%%%%%%%%%%%%%%%%%%%%%%%%%%%%%%%%%%%%%%%%%%%%%%%%%%%%%%%%%%%%%%

We have gathered together in this section some functional analytic remarks used repeatedly in the sequel.

\subsection{Monotone Convergence}
%%%%%%%%%%%%%%%%%%%%%%%%%%%%%%%%%%%%%%%%%%%%%%%%%%%%%%%%%%%%%%%%%%%%%%%%%%%%%%%%%%%%%%%%%%%%%%%%%%%%%%%%%%%%%%%%%%%%%%%%

We recall the analogue of the Beppo Levi monotone convergence theorem for operators in the form convenient for our purpose.

Let $\cH$ be a separable Hilbert space and $0\le T=T^*\in\cL(\fH)$. For each complete orthonormal system $(e_j)_{j\ge 1}$ of $\cH$, set
$$
\Tr_\cH(T)=\|T\|_1:=\sum_{j\ge 1}\la e_j|T|e_j\ra\in[0,+\infty]\,.
$$
See Theorem 2.14 in \cite{Simon}; in particular the expression on the last right hand side of these equalities is independent of the complete orthonormal system $(e_j)_{j\ge 1}$. Then
$$
T\in\cL^1(\fH)\iff\|T\|_1<\infty\,.
$$

\begin{Lem}[Monotone convergence]\lb{L-Monotone}
Consider a sequence $T_n=T_n^*\in\cL^1(\cH)$ such that
$$
0\le T_1\le T_2\le\ldots\le T_n\le\ldots\,,\qquad\text{ and }\sup_{n\ge 1}\la x|T_n|x\ra<\infty\text{ for all }x\in\cH\,.
$$
Then

\noindent
(a) there exists $T=T^*\in\cL(\cH)$ such that $T_n\to T$ weakly as $n\to\infty$, and

\noindent
(b) $\Tr_\cH(T_n)\to\Tr_\cH(T)$ as $n\to\infty$.
\end{Lem}

\begin{proof}
Since the sequence $\la x|T_n|x\ra\in[0,+\infty)$ is nondecreasing for each $x\in\cH$, 
$$
\la x|T_n|x\ra\to\sup_{n\ge 1}\la x|T_n|x\ra=:q(x)\in[0,+\infty)\quad\text{ for all }x\in\cH
$$
as $n\to\infty$. Hence
$$
\la x|T_n|y\ra=\la y|T_n|x\ra\to\tfrac14(q(x+y)-q(x-y)+iq(x-iy)-iq(x+iy))=:b(x,y)\in\bC
$$
as $n\to+\infty$. By construction, $b$ is a nonnegative sesquilinear form on $\cH$. 

Consider, for each $k\ge 0$,
$$
F_k:=\{x\in\cH\text{ s.t. }\la x|T_n|x\ra\le k\text{ for each }n\ge 1\}\,.
$$
The set $F_k$ is closed for each $k\ge 0$, being the intersection of the closed sets defined by the inequality $\la x|T_n|x\ra\le k$ as $n\ge 1$. Since the sequence $\la x|T_n|x\ra$ is bounded for each $x\in\cH$, 
$$
\bigcup_{k\ge 0}F_k=\cH\,.
$$
Applying Baire's theorem shows that there exists $N\ge 0$ such that $\mathring{F}_N\not=\varnothing$. In other words, there exists $r>0$ and $x_0\in\cH$ such that
$$
|x-x_0|\le r\implies|\la x|T_n|x\ra|\le N\text{ for all }n\ge 1\,.
$$
By linearity and positivity of $T_n$, this implies
$$
|\la z|T_n|z\ra|\le \tfrac2r(M+N)\|z\|^2\text{ for all }n\ge 1\,,\quad\text{ with }M:=\sup_{n\ge 1}\la x_0|T_n|x_0\ra\,.
$$
In particular
$$
\sup_{|z|\le 1}q(z)\le\tfrac2r(M+N)\,,\quad\text{ so that }|b(x,y)|\le\frac2r(M+N)|\|x\|_\cH\|y\|_\cH
$$
for each $x,y\in\cH$ by the Cauchy-Schwarz inequality. By the Riesz representation theorem, there exists $T\in\cL(\cH)$ such that
$$
T=T^*\ge 0\,,\quad\text{ and }\quad b(x,y)=\la x|T|y\ra\,.
$$
This proves (a). Observe that $T\ge T_n$ for each $n\ge 1$, so that 
$$
\sup_{n\ge 1}\Tr_\cH(T_n)\le\Tr_\cH(T)\,.
$$
In particular
$$
\sup_{n\ge 1}\Tr_\cH(T_n)=+\infty\implies\Tr_\cH(T)=+\infty\,.
$$
Since the sequence $\Tr_\cH(T_n)$ is nondecreasing,
$$
\Tr_\cH(T_n)\to\sup_{n\ge 1}\Tr_\cH(T_n)\quad\text{ as }n\to\infty\,.
$$
By the noncommutative variant of Fatou's lemma (Theorem 2.7 (d) in \cite{Simon}), 
$$
\sup_{n\ge 1}\Tr_\cH(T_n)<\infty\implies T\in\cL^1(\cH)\text{ and }\Tr_\cH(T)\le\sup_{n\ge 1}\Tr_\cH(T_n)\,.
$$
Since the opposite inequality is already known to hold, this proves (b).
\end{proof}

\smallskip
Here is a convenient variant of this lemma.

\begin{Cor}\lb{C-Monotone}
Consider a sequence $T_n=T_n^*\in\cL^1(\cH)$ such that
$$
0\le T_1\le T_2\le\ldots\le T_n\le\ldots\,,\qquad\text{ and }\sup_{n\ge 1}\Tr_{\cH}(T_n)<\infty\,.
$$
Then there exists $T\in\cL^1(\fH)$ such that $T_n\to T$ weakly as $n\to\infty$, and
$$
T=T^*\ge 0\,,\quad\text{ and }\quad\Tr_\cH(T)=\lim_{n\to\infty}\Tr_\cH(T_n)\,.
$$
\end{Cor}

\begin{proof}
Since any $x\in\cH\setminus\{0\}$ can be normalized and completed into a complete orthonormal system of $\cH$, one has
$$
\sup_{n\ge 1}\la x|T|x\ra\le\|x\|_\cH^2\sup_{n\ge 1}\Tr_{\cH}(T_n)<\infty\,.
$$
One concludes by applying Lemma \ref{L-Monotone} (a) and (b).
\end{proof}

\subsection{Finite Energy Condition}
%%%%%%%%%%%%%%%%%%%%%%%%%%%%%%%%%%%%%%%%%%%%%%%%%%%%%%%%%%%%%%%%%%%%%%%%%%%%%%%%%%%%%%%%%%%%%%%%%%%%%%%%%%%%%%%%%%%%%%%%

In the sequel, we shall repeatedly encounter the following typical situation. Let $A=A^*\ge 0$ be an unbounded self-adjoint operator on $\cH$ with domain $\Dom(A)$, and let $E$ be its spectral decomposition.

Let $T\in\cL^1(\fH)$ satisfy $T=T^*\ge 0$, and let $(e_j)_{j\ge 1}$ be a complete orthonormal system of eigenvectors of $T$ with $Te_j=\tau_je_j$ and $\tau_j\in[0,+\infty)$ for each $j\ge 1$.

\begin{Lem}\lb{L-Energ}
Assume that
\be\lb{TrTA}
\sum_{j\ge 1}\tau_j\int_0^\infty \l\la e_j|E(d\l)|e_j\ra<\infty\,.
\ee
Then
$$
T^{1/2}AT^{1/2}:=\sum_{j,k\ge 1}\tau_j^{1/2}\tau_k^{1/2}\left(\int_0^\infty\l\la e_j|E(d\l)|e_k\ra\right)|e_j\ra\la e_k|
$$
satisfies
$$
0\le T^{1/2}AT^{1/2}=(T^{1/2}AT^{1/2})^*\in\cL^1(\cH)
$$
and 
$$
\Tr_\cH(T^{1/2}AT^{1/2})=\sum_{j\ge 1}\tau_j\int_0^\infty \l\la e_j|E(d\l)|e_j\ra\,.
$$
\end{Lem}

\begin{proof}
For each Borel $\om\subset\bR$ and each $x,y\in\scrH$, one has
$$
|\la x|E(\om)|y\ra|=|\la E(\om)x|E(\om)y\ra|\le\|E(\om)x\|\|E(\om)y\|=\la x|E(\om)|x\ra^{1/2}\la y|E(\om)|y\ra^{1/2}
$$
since $E(\om)$ is a self-adjoint projection. In particular, for each $\a>0$, one has
$$
2|\la x|E(\om)|y\ra|\le\a\la x|E(\om)|x\ra+\tfrac1\a\la y|E(\om)|y\ra\,.
$$
Hence
$$
a_{jk}:=\int_0^\infty\l\la e_j|E(d\l)|e_k\ra\in\bC
$$
and satisfies 
$$
2|a_{jk}|^2\le\a a_{jj}+\tfrac1\a a_{kk}
$$
for all $\a>0$, so that
$$
|a_{jk}|^2\le a_{jj}a_{kk}
$$
for all $j,k\ge 1$. 

Since $(\tau_ja_{jj})_{j\ge 1}\in\ell^1(\bN^*)$ by \eqref{TrTA} and since 
$$
\la e_j|T^{1/2}AT^{1/2}|e_k\ra=\tau_j^{1/2}\tau_k^{1/2}a_{jk}=\overline{\la e_k|T^{1/2}AT^{1/2}|e_j\ra}\,,
$$
one concludes that $T^{1/2}AT^{1/2}=(T^{1/2}AT^{1/2})^*\in\cL^2(\scrH)$. Moreover, for each $x\in\scrH$
$$
\ba
\la x|T^{1/2}AT^{1/2}|x\ra=&\sum_{j,k\ge 1}\tau_j^{1/2}\tau_k^{1/2}\overline{\la e_j|x\ra}\la e_k|x\ra\int_0^\infty\l\la e_j|E(d\l)|e_k\ra
\\
\ge&\int_0^\infty\l\La\sum_{j\ge 1}\tau_j^{1/2}\la e_j|x\ra e_j|E(d\l)| \sum_{j\ge 1}\tau_j^{1/2}\la e_j|x\ra e_j\Ra
\\
=&\int_0^\infty\l\la T^{1/2}x|E(d\l)| T^{1/2}x \ra\ge 0\,,
\ea
$$
so that $T^{1/2}AT^{1/2}\ge 0$.

Finally
$$
\ba
\sum_{l\ge 1}\la e_l|T^{1/2}AT^{1/2}|e_l\ra=\sum_{l\ge 1}\sum_{j,k\ge 1}\tau_j^{1/2}\tau_k^{1/2}\left(\int_0^\infty\l\la e_j|E(d\l)|e_k\ra\right)\ra e_l|e_j\ra\la e_k|e_l\ra
\\
=\sum_{l\ge 1}\sum_{j,k\ge 1}\tau_j^{1/2}\tau_k^{1/2}\left(\int_0^\infty\l\la e_j|E(d\l)|e_k\ra\right)\de_{lj}\de_{lk}=\sum_{l\ge 1}\tau_l\int_0^\infty\l\la e_l|E(d\l)|e_l\ra&<\infty
\ea
$$
so that
$$
\|T^{1/2}AT^{1/2}\|_1=\Tr_{\cH}(T^{1/2}AT^{1/2})=\sum_{l\ge 1}\tau_l\int_0^\infty\l\la e_l|E(d\l)|e_l\ra<\infty
$$
and in particular $T^{1/2}AT^{1/2}\in\cL^1(\cH)$.
\end{proof}

\begin{Cor}\lb{C-Energ}
Let $T\in\cL(\fH)$ satisfy $T=T^*\ge 0$ and \eqref{TrTA}. Let $\Phi_n:\,\bR_+\to\bR_+$ be a sequence of continuous, bounded and nondecreasing functions such that
$$
0\le\Phi_1(r)\le\Phi_2(r)\le\ldots\le\Phi_n(r)\to r\quad\text{ as }n\to\infty\,.
$$
Set
$$
\Phi_n(A):=\int_0^\infty\Phi_n(\l)E(d\l)\in\cL(\cH)\,.
$$
Then $\Phi_n(A)=\Phi_n(A)^*\ge 0$ for each $n\ge 1$ and, for each $T\in\cL^1(\cH)$ such that $T=T^*\ge 0$, the sequence $T^{1/2}\Phi_n(A)T^{1/2}$ converges weakly to $T^{1/2}AT^{1/2}$ as $n\to\infty$. Moreover
$$
\Tr_\cH(T\Phi_n(A))\to\Tr_\cH(T^{1/2}AT^{1/2})\qquad\text{ as }n\to\infty\,.
$$
\end{Cor}

\begin{proof}
Since $E$ is a resolution of the identity on $[0,+\infty)$, and since $\Phi_n$ is continuous, bounded and with values in $[0,+\infty)$, the operators $\Phi_n(A)$ satisfy
$$
0\le\Phi_n(A)=\Phi_n(A)^*\le\left(\sup_{z\ge 0}\Phi_n(z)\right)I_\cH
$$
and
$$
0\le\Phi_1(A)\le\Phi_2(A)\le\ldots\le\Phi_n(A)\le\ldots
$$
Set $R_n:=T^{1/2}\Phi_n(A)T^{1/2}$; by definition $0\le R_n=R_n^*\in\cL(\fH)$ and one has
$$
0\le R_1\le R_2\le\ldots\le R_n\le\ldots
$$
together with
$$
\Tr_\cH(R_n)=\sum_{j\ge 1}\tau_j\int_0^\infty\Phi_n(\l)\la e_j|E(d\l)|e_j\ra\le\sum_{j\ge 1}\tau_j\int_0^\infty\l\la e_j|E(d\l)|e_j\ra<\infty
$$
by \eqref{TrTA}. Applying Corollary \ref{C-Monotone} shows that $R_n$ converges weakly to some $R\in\cL^1(\cH)$ such that $R=R^*\ge 0$. Finally
$$
T^{1/2}AT^{1/2}-R_n=\sum_{j,k\ge 1}\tau_j^{1/2}\tau_k^{1/2}\left(\int_0^\infty(\l-\Phi_n(\l))\la e_j|E(d\l)|e_k\ra\right)|e_j\ra\la e_k|
$$
so that
$$
\ba
\la x|T^{1/2}AT^{1/2}-R_n|x\ra=&\int_0^\infty(\l-\Phi_n(\l))\La\sum_{j\ge 1}\tau_j^{1/2}\la e_j|x\ra e_j|E(d\l)|\sum_{k\ge 1}\tau_k^{1/2}\la e_k|x\ra e_k\Ra
\\
=&\int_0^\infty(\l-\Phi_n(\l))\la T^{1/2}x|E(d\l)|T^{1/2}x\ra\ge 0\,.
\ea
$$
Hence
$$
0\le T^{1/2}AT^{1/2}-R_n=(T^{1/2}AT^{1/2}-R_n)^*\in\cL^1(\cH)
$$
so that
$$
\ba
\|T^{1/2}AT^{1/2}-R_n\|_1=&\Tr_\cH(T^{1/2}AT^{1/2}-R_n)
\\
=&\sum_{j\ge 1}\tau_j\int_0^\infty(\l-\Phi_n(\l))\la e_j|E(d\l)|e_j\ra\to 0
\ea
$$
as $n\to\infty$ by monotone convergence. Hence $R_n\to T^{1/2}AT^{1/2}$ in $\cL^1(\cH)$ and one has in particular
$$
\Tr_{\cH}(T\Phi_n(A))=\Tr_{\cH}(T^{1/2}\Phi_n(A)T^{1/2})\to\Tr_{\cH}(T^{1/2}AT^{1/2})\,.
$$
\end{proof}

\subsection{Energy and Partial Trace}
%%%%%%%%%%%%%%%%%%%%%%%%%%%%%%%%%%%%%%%%%%%%%%%%%%%%%%%%%%%%%%%%%%%%%%%%%%%%%%%%%%%%%%%%%%%%%%%%%%%%%%%%%%%%%%%%%%%%%%%%

Let $\cH_1$ and $\cH_2$ be two separable Hilbert spaces. Let $A=A^*\ge 0$ be an unbounded self-adjoint operator on $\cH_1$ with domain $\Dom(A)$, and let $E$ be its spectral decomposition. Let $S\in\cL^1(\cH_1)$ satisfy
$S=S^*\ge 0$, and let $(e_j)_{j\ge 1}$ be a complete orthonormal system of $\cH_1$ of eigenvectors of $S$, with eigenvalues $(\si_j)_{j\ge 1}$ such that $Se_j=\si_je_j$ for each $j\ge 1$. Assume that
$$
\sum_{j\ge 1}\si_j\int_0^{+\infty}\l\la e_j|E(d\l)|e_j\ra<\infty\,.
$$

\begin{Lem}
Let $T\in\cL^1(\cH_1\otimes\cH_2)$ satisfy the partial trace condition
$$
\Tr(T|\cH_2)=S\,.
$$
Then $T^{1/2}(A\otimes I_{\cH_2})T^{1/2}\in\cL^1(\cH_1\otimes\cH_2)$ and
$$
\Tr_{\cH_1\otimes\cH_2}(T^{1/2}(A\otimes I)T^{1/2})=\Tr_{\cH_1}(S^{1/2}AS^{1/2})\,.
$$
\end{Lem}

\begin{proof}
For all $n\ge 1$, set $A_n=\Phi_n(A)\in\cL(\fH)$, with 
$$
\Phi_n(r):=\frac{r}{1+\frac1nr}\,,\qquad\text{ for all }r\ge 0\,.
$$
By construction, one has 
$$
A_n=A_n^*\ge 0\quad\text{ and }\quad A_1\le A_2\le\ldots\le A_n\le\ldots
$$
Hence $T^{1/2}(A_n\otimes I_{\cH_2})T^{1/2}=(T^{1/2}(A_n\otimes I_{\cH_2})T^{1/2})^*\ge 0$ for all $n\ge 1$, and
$$
T^{1/2}(A_1\otimes I_{\cH_2})T^{1/2}\le T^{1/2}(A_2\otimes I_{\cH_2})T^{1/2}\le\ldots\le T^{1/2}(A_n\otimes I_{\cH_2})T^{1/2}\le\ldots
$$
and since 
$$
\Tr_{\cH_1\otimes\cH_2}(T(A_n\otimes I_{\cH_2}))=\Tr_{\cH_1}(SA_n)\to\Tr_{\cH_1}(S^{1/2}AS^{1/2})
$$
as $n\to\infty$ by the partial trace condition and Corollary \ref{C-Energ}, we conclude from Corollary \ref{C-Monotone} that
$$
T^{1/2}(A\otimes I_{\cH_2})T^{1/2}=(T^{1/2}(A\otimes I_{\cH_2})T^{1/2})^*\ge 0
$$
and that
$$
\Tr_{\cH_1\otimes\cH_2}(T^{1/2}(A\otimes I)T^{1/2})=\Tr_{\cH_1}(S^{1/2}AS^{1/2})\,.
$$
\end{proof}

%%%%%%%%%%%%%%%%%%%%%%%%%%%%%%%%%%%%%%%%%%%%%%%%%%%%%%%%%%%%%%%%%%%%%%%%%%%%%%%%%%%%%%%%%%%%%%%%%%%%%%%%%%%%%%%%%%%%%%%%
\section{Couplings}
%%%%%%%%%%%%%%%%%%%%%%%%%%%%%%%%%%%%%%%%%%%%%%%%%%%%%%%%%%%%%%%%%%%%%%%%%%%%%%%%%%%%%%%%%%%%%%%%%%%%%%%%%%%%%%%%%%%%%%%%

Let $\fH:=L^2(\bR^d)$. An operator $R\in\cL(\fH)$ is a density operator if 
$$
R=R^*\ge 0\quad\hbox{ and }\quad\Tr(R)=1\,.
$$
We denote by $\cD(\fH)$ the set of density operators on $\fH$, and define
$$
\cD_2(\fH):=\{R\in\cD(\fH)\hbox{ s.t. }\Tr(R^{1/2}(|y|^2-\Dlt_y)R^{1/2})<\infty\}\,.
$$
The set of Borel probability measures on $\bR^d\times\bR^d$ is denoted by $\cP(\bR^d\times\bR^d)$. We denote by $\cP_2(\bR^d\times\bR^d)$ the set of Borel probability measures $\mu$ on $\bR^d\times\bR^d$ such that
$$
\iint_{\bR^d\times\bR^d}(|x|^2+|\xi|^2)\mu(dxd\xi)<\infty\,.
$$
The set of Borel probability measures on $\bR^d\times\bR^d$ which are absolutely continuous with respect to the Lebesgue measure on $\bR^d\times\bR^d$ is denoted $\cP^{ac}(\bR^d\times\bR^d)\subset\cP(\bR^d\times\bR^d)$.
We set $\cP^{ac}_2(\bR^d\times\bR^d)=\cP^{ac}(\bR^d\times\bR^d)\cap\cP_2(\bR^d\times\bR^d)$, and we identify elements of $\cP^{ac}(\bR^d\times\bR^d)$ with their densities with respect to the Lebesgue measure.

Let $R_1,R_2\in\cD(\fH)$; a coupling of $R_1$ and $R_2$ is an element $\cR\in\cD(\fH\otimes\fH)$ such that
$$
\Tr_{\fH\otimes\fH}((A\otimes I+I\otimes B)\cR)=\Tr_\fH(R_1A)+\Tr_\fH(R_2B)\,.
$$
The set of couplings of $R_1$ and $R_2$ will be denoted by $\cC(R_1,R_2)$. Obviously the tensor product $R_1\otimes R_2\in\cC(R_1,R_2)$, so that $\cC(R_1,R_2)\not=\varnothing$.

Let $f$ be a probability density on $\bR^d\times\bR^d$, and let $R\in\cD(\fH)$. A coupling of $f$ and $R$ is an ultraweakly measurable operator-valued function $(x,\xi)\mapsto Q(x,\xi)$ defined a.e. on $\bR^d\times\bR^d$ with 
values in $\cL(\fH)$ such that
$$
\ba
Q(x,\xi)=Q(x,\xi)^*\ge 0\,,\quad\iint_{\bR^d\times\bR^d}Q(x,\xi)dxd\xi=R
\\
\hbox{ and }\Tr_\fH(Q(x,\xi))=f(x,\xi)\hbox{Ê for a.e. }(x,\xi)\in\bR^d\times\bR^d&\,.
\ea
$$
The set of couplings of $f$ and $R$ will also be denoted by $\cC(f,R)$. Since the map $(x,\xi)\mapsto f(x,\xi)R$ (henceforth denoted $f\otimes R$) obviously belongs to $\cC(f,R)$, one has $\cC(f,R)\not=\varnothing$.

\smallskip
In general, one does not know much about the general structure of couplings between two density operators. However, the case where one of the density operators is a rank $1$ projection is particularly simple.

\begin{Lem}\lb{L-rank1}
Let $P=P^*\in\cL(\fH)$ be a rank $1$ projection. Then

\noindent
(i) for each probability density $f$ on $\bR^d\times\bR^d$, one has $\cC(f,P)=\{f\otimes P\}$;

\noindent
(ii) for each $R\in\cD(\fH)$, one has $\cC(P,R)=\{P\otimes R\}$ and $\cC(R,P)=\{R\otimes P\}$.
\end{Lem}

\smallskip
This is in complete analogy with the following elementary observation: if $\mu\in\cP(\bR^d)$ and $y_0\in\bR^d$, the only coupling of $\mu$ and $\de_{y_0}$ is $\mu\otimes\de_{y_0}$. In other words, self-adjoint rank-$1$ projections 
are the quantum analogue of points in this picture. 

\begin{proof}
Let $Q\in\cC(f,P)$; one has
$$
\ba
\iint_{\bR^d\times\bR^d}\Tr_\fH((I-P)Q(x,\xi)(I-P))dxd\xi
\\
=\Tr_\fH\left((I-P)\iint_{\bR^d\times\bR^d}Q(x,\xi)dxd\xi(I-P\right)
\\
=\Tr_\fH((I-P)P(I-P))=0&\,.
\ea
$$
Since $(I-P)Q(x,\xi)(I-P)\ge 0$ for a.e. $(x,\xi)\in\bR^d\times\bR^d$, this implies that
$$
(I-P)Q(x,\xi)(I-P)=0\quad\hbox{ for a.e. }(x,\xi)\in\bR^d\times\bR^d\,.
$$
Since $Q(x,\xi)=Q(x,\xi)^*\ge 0$ for a.e. $(x,\xi)\in\bR^d\times\bR^d$, we deduce from the Cauchy-Schwarz inequality that, for all $\phi,\psi\in\fH$
$$
\ba
|\la P\phi|Q(x,\xi)|(I-P)\psi\ra|^2=|\la(I-P)\psi|Q(x,\xi)|P\phi\ra|^2
\\
\le\la P\phi|Q(x,\xi)|P\phi\ra^{1/2}\la(I-P)\psi|Q(x,\xi)|(I-P)\psi\ra^{1/2}=0&\,.
\ea
$$
Hence $(I-P)Q(x,\xi)P=PQ(x,\xi)(I-P)=0$ for a.e. $(x,\xi)\in\bR^d\times\bR^d$, so that
$$
Q(x,\xi)=PQ(x,\xi)P\quad\hbox{ for a.e. }(x,\xi)\in\bR^d\times\bR^d\,.
$$
Writing $P$ as $P=|u\ra\la u|$ where $u\in\fH$ is a unit vector, we conclude that
$$
Q(x,\xi)=\la u|Q(x,\xi)|u\ra P\quad\hbox{ for a.e. }(x,\xi)\in\bR^d\times\bR^d\,.
$$
Finally
$$
\Tr(Q(x,\xi))=f(x,\xi)=\la u|Q(x,\xi)|u\ra\quad\hbox{ for a.e. }(x,\xi)\in\bR^d\times\bR^d\,.
$$
This concludes the proof of (i).

As for (ii), let $\cQ\in\cC(R,P)$. Then
$$
\Tr_{\fH\otimes\fH}((I\otimes(I-P))\cQ(I\otimes(I-P)))=\Tr_\fH(I-P)P(I-P))=0\,.
$$
Hence
$$
(I\otimes(I-P))\cQ(I\otimes(I-P)=0\,.
$$
Since $\cQ=\cQ^*\ge 0$, the Cauchy-Schwarz inequality implies that, for all $\phi,\phi',\psi,\psi'\in\fH$
$$
\ba
|\la\phi\otimes\psi|(I\otimes P)\cQ(I\otimes(I-P))|\phi'\otimes\psi'\ra|=|\la\phi'\otimes\psi|'(I\otimes(I-P))\cQ(I\otimes P)|\phi\otimes\psi\ra|
\\
\le\la\phi\otimes\psi|(I\otimes P)\cQ(I\otimes P)|\phi\otimes\psi\ra^{1/2}\la\phi'\otimes\psi'|(I\otimes(I-P))\cQ(I\otimes(I-P))|\phi'\otimes\psi'\ra^{1/2}
\ea
$$
so that
$$
(I\otimes P)\cQ(I\otimes(I-P))=(I\otimes(I-P))\cQ(I\otimes P)=0\,.
$$
Hence
$$
\cQ=(I\otimes P)\cQ(I\otimes P)\,.
$$
Writing $P=|u\ra\la u|$ with $u\in\fH$ and $|u|=1$ as above, we conclude that 
$$
\la\phi\otimes\psi|\cQ|\phi'\otimes\psi'\ra=\la\phi\otimes u|\cQ|\phi'\otimes u\ra\la u|\psi\ra\la u|\psi'\ra\,.
$$
This shows that $\cQ=L\otimes|u\ra\la u|=L\otimes P$, where $L=L^*$ is the element of $\cL(\fH)$ such that
$$
\la\phi|L|\phi'\ra=\la\phi\otimes u|\cQ|\phi'\otimes u\ra
$$
for each $\phi,\phi'\in\fH$. (Observe indeed that $(\phi,\phi')\mapsto\la\phi\otimes u|\cQ|\phi'\otimes u\ra$ is a continuous, symmetric bilinear functional on $\fH$, and is therefore represented by a unique self-adjoint element of $\cL(\fH)$.)
We conclude by observing that
$$
\Tr_{\fH\otimes\fH}((A\otimes I)\cQ)=\Tr_\fH(AR)=\Tr_\fH(LR)
$$
for each finite rank operator $A\in\cL(\fH)$, and this implies that $\cQ=R\otimes P$. 

The case of $\cQ'\in\cC(P,R)$ is handled similarly.
\end{proof}

\smallskip
Next we explain how to ``disintegrate'' a coupling with respect to one of its marginals when this marginal is a probability density.

\begin{Lem}\lb{L-Disint}
Let $f\in\cP^{ac}(\bR^d\times\bR^d)$, let $R\in\cD(\fH)$ and let $Q\in\cC(f,R)$. There exists a $\si(\cL^1(\fH),\cL(\fH))$ weakly measurable function $(x,\xi)\mapsto Q_f(x,\xi)$ defined a.e. on $\bR^d\times\bR^d$ with values in $\cL^1(\fH)$ 
such that
$$
Q_f(x,\xi)=Q^*_f(x,\xi)\ge 0\,,\quad\Tr(Q_f(x,\xi))=1\,,\quad\hbox{ and }Q(x,\xi)=f(x,\xi)Q_f(x,\xi)
$$
for a.e. $(x,\xi)\in\bR^d\times\bR^d$.
\end{Lem}

\begin{proof}
Let $f_1$ be a Borel measurable function defined on $\bR^d\times\bR^d$ and such that $f(x,\xi)=f_1(x,\xi)$ for a.e. $(x,\xi)\in\bR^d\times\bR^d$. Let $\cN$ be the Borel measurable set defined as follows: 
$\cN:=\{(x,\xi)\in\bR^d\times\bR^d\hbox{ s.t. }f(x,\xi)=0\}$, and let $u\in\fH$ satisfy $|u|=1$.
Consider the function
$$
(x,\xi)\mapsto Q_f(x,\xi):=\frac{Q(x,\xi)+\indc_\cN(x,\xi)|u\ra\la u|}{f_1(x,\xi)+\indc_\cN(x,\xi)}\in\cL(\fH)
$$
defined a.e. on $\bR^d\times\bR^d$. The function $f_1+\indc_{\cN}>0$ is Borel measurable on $\bR^d\times\bR^d$ while $(x,\xi)\mapsto\la\phi|Q(x,\xi)|\psi\ra$ is measurable and defined a.e. on $\bR^d\times\bR^d$ for each $\phi,\psi\in\fH$.
Set $\cA:\,\cL(\fH)\times(0,+\infty)\ni(T,\l)\mapsto\l^{-1}T\in\cL(\fH)$; since $\cA$ is continuous, the function $Q_f:=\cA(Q+\indc_\cN\otimes|u\ra\la u|,f_1+\indc_\cN)$ is weakly measurable on $\bR^d\times\bR^d$. Since $f_1+\indc_\cN>0$, 
and since $Q(x,\xi)=Q^*(x,\xi)\ge 0$, one has $(Q(x,\xi)+\indc_\cN\otimes|u\ra\la u|)^*=Q(x,\xi)+\indc_\cN\otimes|u\ra\la u|\ge 0$ for a.e. $(x,\xi)\in\bR^d\times\bR^d$. On the other hand, for a.e. $(x,\xi)\in\bR^d\times\bR^d$, one has
$\Tr(Q(x,\xi)+\indc_\cN\otimes|u\ra\la u|)=f(x,\xi)+\indc_\cN(x,\xi)$, so that $\Tr(Q_f(x,\xi))=1$. Finally
$$
f(x,\xi)Q_f(x,\xi)=\frac{f(x,\xi)Q(x,\xi)}{f_1(x,\xi)+\indc_\cN(x,\xi)}=Q(x,\xi)\quad\hbox{ for a.e. }(x,\xi)\in\bR^d\times\bR^d\,,
$$
since $f=f_1$ a.e. on $\bR^d\times\bR^d$ and $\indc_\cN(x,\xi)=0$ for a.e. $(x,\xi)\in\bR^d\times\bR^d$ such that $f(x,\xi)>0$. Since $Q_f$ satisfies $\Tr(Q_f(x,\xi))=1$ for a.e. $(x,\xi)\in\bR^d\times\bR^d$ and is weakly measurable on 
$\bR^d\times\bR^d$, it is $\si(\cL^1(\fH),\cL(\fH))$ weakly measurable.
\end{proof}

%%%%%%%%%%%%%%%%%%%%%%%%%%%%%%%%%%%%%%%%%%%%%%%%%%%%%%%%%%%%%%%%%%%%%%%%%%%%%%%%%%%%%%%%%%%%%%%%%%%%%%%%%%%%%%%%%%%%%%%
\section{Triangle Inequalities}
%%%%%%%%%%%%%%%%%%%%%%%%%%%%%%%%%%%%%%%%%%%%%%%%%%%%%%%%%%%%%%%%%%%%%%%%%%%%%%%%%%%%%%%%%%%%%%%%%%%%%%%%%%%%%%%%%%%%%%%

The following ``pseudo metrics'' have been defined in \cite{FGMouPaul} and in \cite{FGPaul} respectively.

\begin{Def}
For all $R,S\in\cD_2(\fH)$ and all $f\in\cP_2^{ac}(\bR^d\times\bR^d)$, we set
$$
MK_\hb(R,S):=\inf_{A\in\cC(R,S)}\Tr_{\fH\otimes\fH}(A^{1/2}CA^{1/2})^{1/2}
$$
where
$$
C:=C(x,y,\hb D_x,\hb D_y)=|x-y|^2+|\hb D_x-\hb D_y|^2\,.
$$
Similarly, we set
$$
\cE_\hb(f,R):=\inf_{a\in\cC(f,R)}\left(\iint_{\bR^d\times\bR^d}\Tr_\fH(a(x,\xi)^{1/2}c(x,\xi)a(x,\xi)^{1/2})dxd\xi\right)^{1/2}
$$
where
$$
c(x,\xi):=c(x,\xi,y,\hb D_y)=|x-y|^2+|\xi-\hb D_y|^2\,.
$$
\end{Def}

\smallskip
The above ``pseudometrics'' satisfy the following inequalities.

\begin{Thm}\lb{T-IneqT} Let $f,g\in\cP_2^{ac}(\bR^d\times\bR^d)$, and let $R_1,R_2,R_3\in\cD_2(\fH)$. The following inequalities hold true:

\noindent
(a) $\cE_\hb(f,R_1)\le\MKd(f,g)+\cE_\hb(g,R_1)$;

\noindent
(b) $MK_\hb(R_1,R_3)\le\cE_\hb(f,R_1)+\cE_\hb(f,R_3)$;

\noindent
(c) if $\text{rank}(R_2)=1$, then
$$
MK_\hb(R_1,R_3)\le MK_\hb(R_1,R_2)+MK_\hb(R_2,R_3)\,,
$$
(d) if $\text{rank}(R_2)=1$, then
$$
\MKd(f,g)\le\cE_\hb(f,R_2)+\cE_\hb(g,R_2)\,,
$$ 
(e) if $\text{rank}(R_2)=1$, then
$$
\cE_\hb(f,R_3)\le\cE_\hb(f,R_2)+MK_\hb(R_2,R_3)\,.
$$
\end{Thm}

\smallskip
The proofs of all these triangle inequalities make use of some inequalities between the (classical and/or quantum) transportation cost operators. We begin with an elementary, but useful lemma, which can be viewed as the Peter-Paul
inequality for operators.

\begin{Lem}\lb{L-PP}
Let $T,S$ be unbounded self-adjoint operators on $\fH=L^2(\bR^n)$, with domains $\Dom(T)$ and $\Dom(S)$ respectively such that $\Dom(T)\cap\Dom(S)$ is dense in $\fH$. Then, for all $\a>0$, one has
$$
\la v|TS+ST|v\ra\le\a\la v|T^2|v\ra+\frac1\a\la v|S^2|v\ra\,,\qquad\text{ for all }v\in\Dom(T)\cap\Dom(S)\,.
$$
\end{Lem}

\begin{proof}
Indeed, for each $\a>0$ and each $v\in\Dom(T)\cap\Dom(S)$, one has
$$
\ba
\a\la v|T^2|v\ra+\tfrac1\a\la v|S^2|v\ra-\la v|TS+ST|v\ra
\\
=|\sqrt{\a} Tv|^2+|\tfrac1{\sqrt{\a}}Sv|^2-\la\sqrt{\a}Tv|\tfrac1{\sqrt{\a}}Sv\ra-\la\tfrac1{\sqrt{\a}}Sv|\sqrt{\a}Tv\ra
\\
=\left|\sqrt{\a} Tv-\tfrac1{\sqrt{\a}}Sv\right|^2\ge 0&\,.
\ea
$$
\end{proof}

\begin{Lem}\lb{L-IneqCost}
For each $x,\xi,y,\eta,z\in\bR^d$ and each $\a>0$, one has
$$
\ba
c(x,\xi;z,\hb D_z)\le(1+\a)(|x-y|^2+|\xi-\eta|^2)+(1+\tfrac1\a)c(y,\eta;z,\hb D_z)\,,
\\
C(x,z,\hb D_x,\hb D_z)\le(1+\a)c(y,\eta;x,\hb D_x)+(1+\tfrac1\a)c(y,\eta;z,\hb D_z)\,,
\\
C(x,z,\hb D_x,\hb D_z)\le(1+\a)C(x,y;\hb D_x,\hb D_y)+(1+\tfrac1\a)C(y,z;\hb D_y,\hb D_z)\,,
\\
|x-z|^2+|\xi-\zeta|^2\le(1+\a)c(x,\xi;y,\hb D_y)+(1+\a)c(z,\zeta;y,\hb D_y)\,,
\\
c(x,\xi;z,\hb D_z)\le(1+\a)c(x,\xi;y,\hb D_y)+(1+\tfrac1\a)C(y,z;\hb D_y,\hb D_z)\,.
\ea
$$
\end{Lem}

\smallskip
All these inequalities are of the form $A\le B$ where $A$ and $B$ are unbounded self-adjoint operators on $L^2(\bR^n)$ for some $n\ge 1$, with 
$$
\cW:=\{\psi\in H^1(\bR^n)\text{ s.t. }|x|\psi\in\cH\}\subset\Dom_f(A)\cap\Dom_f(B)\,,
$$
denoting by $\Dom_f(A)$ (resp. $\Dom_f(B)$) the form-domain of $A$ (resp. of $B$) --- see \S VIII.6 in \cite{RS1} on pp. 276--277. The inequality $A\le B$ means that the bilinear form associated to $B-A$ is nonnegative, i.e. that
$$
\la w|A|w\ra\le\la w|B|w\ra\,,\qquad\text{ for all }w\in\cW\,.
$$

\begin{proof}
All these inequalities are proved in the same way. Let us prove for instance the third inequality:
$$
\ba
C(x,z,\hb D_x,\hb D_z)=&|x-y+y-z|^2+|\hb D_x-\hb D_y+\hb D_y-\hb D_z|^2
\\
=&C(x,y;\hb D_x,\hb D_y)+C(y,z;\hb D_y,\hb D_z)
\\
&+2(x-y)\cdot(y-z)+2(\hb D_x-\hb D_y)\cdot(\hb D_y-\hb D_z)\,.
\ea
$$
Observe indeed that the multiplication operators by $(x-y)$ and by $(y-z)$ commute; likewise $(\hb D_x-\hb D_y)$ and $(\hb D_y-\hb D_z)$ commute. By Lemma \ref{L-PP}
$$
\ba
2(x-y)\cdot(y-z)+2(\hb D_x-\hb D_y)\cdot(\hb D_y-\hb D_z)
\\
\le\a C(x,y;\hb D_x,\hb D_y)+\frac1\a C(y,z;\hb D_y,\hb D_z)&\,,
\ea
$$
which concludes the proof of the third inequality. 
\end{proof}

\begin{proof}[Proof of Theorem \ref{T-IneqT} (a)]
By Theorem 2.12 in chapter 2 of \cite{VillaniAMS}, there exists an optimal coupling for $W_2(f,g)$, of the form $f(x,\xi)\de_{\grad\Phi(x,\xi)}(dyd\eta)$, where $\Phi$ is a convex function on $\bR^d\times\bR^d$. Let $Q\in\cC(g,R_1)$ and set 
$$
P(x,\xi;dyd\eta):=f(x,\xi)\de_{\grad\Phi(x,\xi)}(dyd\eta)Q_g(y,\eta)\,,
$$
where $Q_g$ is the disintegration of $Q$ with respect to $f$ obtained in Lemma \ref{L-Disint}. Then $P$ is a nonnegative,self-adjoint operator-valued measure satisfying
$$
\Tr_\fH(P(x,\xi;dyd\eta))=f(x,\xi)\de_{\grad\Phi(x,\xi)}(dyd\eta)
$$
while
$$
\int Pdxd\xi=(\grad\Phi\#f)(y,\eta)dyd\eta Q_g(y,\eta)=g(y,\eta)Q_g(y,\eta)dyd\eta=Q(y,\eta)dyd\eta\,.
$$
In particular
\be\lb{PCfR1}
\int P(x,\xi;dyd\eta)=f(x,\xi)Q_g(\grad\Phi(x,\xi))\in\cC(f,R_1)\,.
\ee
Therefore
$$
\ba
\cE_\hb(f,R_1)^2\le\int\Tr_\fH(Q_g(\grad\Phi(x,\xi))^{1/2}c_\hb(x,\xi)Q_g(\grad\Phi(x,\xi))^{1/2})f(x,\xi)dxd\xi\,.
\ea
$$
By the first inequality in Lemma \ref{L-IneqCost}, one has
$$
c_\hb(x,\xi;z,\hb D_z)\le(1+\a)|(x,\xi)-\grad\phi(x,\xi)|^2+(1+\tfrac1\a)c_\hb(\grad\Phi(x,\xi);z,\hb D_z)
$$
for a.e. $(x,\xi)\in\bR^d\times\bR^d$ and all $\a>0$. Since $g\in\cP_2^{ac}(\bR^d\times\bR^d)$ and $R_1\in\cD_2(\fH)$ and $Q\in\cC(g,R_1)$, then
$$
\ba
\int\Tr_\fH(Q(y,\eta)^{1/2}c_\hb(y,\eta)Q(y,\eta)^{1/2})dyd\eta&
\\
=\int\Tr_\fH(Q_g(\grad\Phi(x,\xi))^{1/2}c_\hb(\grad\Phi(x,\xi))Q_g(\grad\Phi(x,\xi))^{1/2})f(x,\xi)dxd\xi&<\infty\,.
\ea
$$
For each $\eps>0$, set
$$
c_\hb^\eps(x,\xi;z,\hb D_z)=(I+\eps c_\hb(x,\xi;z,\hb D_z))^{-1}c_\hb(x,\xi;z,\hb D_z)\le c_\hb(x,\xi;z,\hb D_z)\,.
$$
Then, for a.e. $(x,\xi)\in\bR^d\times\bR^d$ and each $\eps>0$, one has
$$
Q_g(\grad\Phi(x,\xi))^{1/2}c_\hb(\grad\Phi(x,\xi);z,\hb D_z)Q_g(\grad\Phi(x,\xi))^{1/2}\in\cL^1(\fH)\,,
$$
and
$$
\ba
Q_g(\grad\Phi(x,\xi))^{1/2}&c^\eps_\hb(x,\xi;z,\hb D_z)Q_g(\grad\Phi(x,\xi))^{1/2}
\\
\le&(1+\a)|(x,\xi)-\grad\phi(x,\xi)|^2Q_g(\grad\Phi(x,\xi))
\\
&+(1+\tfrac1\a)Q_g(\grad\Phi(x,\xi))^{1/2}c_\hb(\grad\Phi(x,\xi);z,\hb D_z)Q_g(\grad\Phi(x,\xi))^{1/2}\,.
\ea
$$
Integrating both sides of this inequality with respect to the probability distribution $f(x,\xi)$, one finds
$$
\ba
\int\Tr_\fH(Q_g(\grad\Phi(x,\xi))^{1/2}c^\eps_\hb(x,\xi)Q_g(\grad\Phi(x,\xi))^{1/2})f(x,\xi)dxd\xi
\\
\le(1+\a)\int|(x,\xi)-\grad\phi(x,\xi)|^2f(x,\xi)dxd\xi
\\
+(1+\tfrac1\a)\int\Tr_\fH(Q_g(\grad\Phi(x,\xi))^{1/2}c_\hb(\grad\Phi(x,\xi))Q_g(\grad\Phi(x,\xi))^{1/2})f(x,\xi)dxd\xi
\\
\le(1+\a)\MKd(f,g)^2
\\
+(1+\tfrac1\a)\int\Tr_\fH(Q_g(y,\eta)^{1/2}c_\hb(y,\eta)Q_g(y,\eta)^{1/2})g(y,\eta)dyd\eta
\\
\le(1+\a)\MKd(f,g)^2+(1+\tfrac1\a)\int\Tr_\fH(Q(y,\eta)^{1/2}c_\hb(y,\eta)Q(y,\eta)^{1/2})dyd\eta&\,.
\ea
$$
Minimizing the last right hand side of this inequality in $Q\in\cC(g,R_1)$ shows that
$$
\ba
\int\Tr_\fH(Q_g(\grad\Phi(x,\xi))^{1/2}c^\eps_\hb(x,\xi)Q_g(\grad\Phi(x,\xi))^{1/2})f(x,\xi)dxd\xi
\\
\le(1+\a)\MKd(f,g)^2+(1+\tfrac1\a)\cE(g,R_1)^2&\,.
\ea
$$
Passing to the limit as $\eps\to 0^+$ in the left hand side and applying Corollary \ref{C-Energ} shows that
$$
\ba
\cE_\hb(f,R_1)^2\le&\int\Tr_\fH(Q_g(\grad\Phi(x,\xi))^{1/2}c_\hb(x,\xi)Q_g(\grad\Phi(x,\xi))^{1/2})f(x,\xi)dxd\xi
\\
\le&(1+\a)\MKd(f,g)^2+(1+\tfrac1\a)\cE(g,R_1)^2\,,
\ea
$$
the first inequality being a consequence of the definition of $\cE_\hb$ according to \eqref{PCfR1}.

Finally, minimizing the right hand side of this inequality as $\a>0$, i.e. choosing $\a=\cE_\hb(f,g)/\MKd(f,g)$ if $f\not=g$ a.e. on $\bR^d\times\bR^d$, or letting $\a\to+\infty$ if $f=g$, we arrive at the inequality
$$
\ba
\cE_\hb(f,R_1)^2\le&\MKd(f,g)^2+\cE_\hb(g,R_1)^2+2\cE_\hb(g,R_1)\MKd(f,g)
\\
=&(\MKd(f,g)+\cE_\hb(g,R_1))^2\,,
\ea
$$
which is precisely the inequality (a).
\end{proof}

\smallskip
\begin{proof}[Proof of Theorem \ref{T-IneqT} (b)]
Let $Q_1\in\cC(f,R_1)$ and $Q_3\in\cC(f,R_3)$. Let $Q_{1,f}$ and $Q_{3,f}$ be the disintegrations of $Q_1$ and $Q_3$ with respect to $f$ obtained in Lemma \ref{L-Disint}. For each $\eps>0$, set
\be\lb{Cheps}
C_\hb^\eps(x,z,\hb D_x,\hb D_z)=(I+\eps C_\hb(x,z,\hb D_x,\hb D_z))^{-1}C_\hb(x,z,\hb D_x,\hb D_z)
\ee
and observe that 
$$
0\le C_\hb^\eps(x,z,\hb D_x,\hb D_z)=C_\hb^\eps(x,z,\hb D_x,\hb D_z)^*\in\cL(\cH\otimes\fH)\,.
$$
By the second inequality in Lemma \ref{L-IneqCost}, for all $(y,\eta)\in\bR^d\times\bR^d$ and all $\a>0$, one has
$$
\ba
C_\hb^\eps(x,z,\hb D_x,\hb D_z)\le C_\hb(x,z,\hb D_x,\hb D_z)
\\
\le (1+\a)c_\hb(y,\eta;x,\hb D_x)+(1+\tfrac1\a)c_\hb(y,\eta;z,\hb D_z)&\,.
\ea
$$
Therefore, for a.e. $(y,\eta)\in\bR^d\times\bR^d$, one has
$$
\ba
(Q_{1,f}(y,\eta)\otimes Q_{3,f}(y,\eta))^{1/2}C_\hb^\eps(x,z,\hb D_x,\hb D_z)(Q_{1,f}(y,\eta)\otimes Q_{3,f}(y,\eta))^{1/2}
\\
\le(1+\a)(Q_{1,f}(y,\eta)\otimes Q_{3,f}(y,\eta))^{1/2}c_\hb(y,\eta;x,\hb D_x)(Q_{1,f}(y,\eta)\otimes Q_{3,f}(y,\eta))^{1/2}
\\
+(1+\tfrac1\a)(Q_{1,f}(y,\eta)\otimes Q_{3,f}(y,\eta))^{1/2}c_\hb(y,\eta;z,\hb D_z)(Q_{1,f}(y,\eta)\otimes Q_{3,f}(y,\eta))^{1/2}
\\
=(1+\a)\left(Q_{1,f}(y,\eta)^{1/2}c_\hb(y,\eta;x,\hb D_x)Q_{1,f}(y,\eta)^{1/2}\right)\otimes Q_{3,f}(y,\eta)
\\
+(1+\tfrac1\a)Q_{1,f}(y,\eta)\otimes\left(Q_{3,f}(y,\eta)^{1/2}c_\hb(y,\eta;z,\hb D_z)Q_{1,f}(y,\eta)^{1/2}\right)\,.
\ea
$$
Taking the trace in $\fH\otimes\fH$ of both sides of this inequality shows that
$$
\ba
\Tr_{\fH\otimes\fH}\left((Q_{1,f}\otimes Q_{3,f}(y,\eta))C_\hb^\eps(x,z,\hb D_x,\hb D_z)\right)
\\
=
\Tr_{\fH\otimes\fH}\left((Q_{1,f}\otimes Q_{3,f}(y,\eta))^{1/2}C_\hb^\eps(x,z,\hb D_x,\hb D_z)(Q_{1,f}\otimes Q_{3,f}(y,\eta))^{1/2}\right)
\\
\le(1+\a)\Tr_\fH\left(Q_{1,f}(y,\eta)^{1/2}c_\hb(y,\eta;x,\hb D_x)Q_{1,f}(y,\eta)^{1/2}\right)
\\
+(1+\tfrac1\a)\Tr_\fH\left(Q_{3,f}(y,\eta)^{1/2}c_\hb(y,\eta;z,\hb D_z)Q_{1,f}(y,\eta)^{1/2}\right)
\ea
$$
for a.e. $(y,\eta)\in\bR^d\times\bR^d$. Integrating both sides of this inequality in $(y,\eta)$ with respect to $f$ shows that
$$
\ba
\Tr_{\fH\otimes\fH}\left(\left(\int(Q_{1,f}\otimes Q_{3,f}(y,\eta))f((y,\eta)dyd\eta\right)C_\hb^\eps(x,z,\hb D_x,\hb D_z)\right)
\\
\le(1+\a)\int\Tr_\fH\left(Q_{1,f}(y,\eta)^{1/2}c_\hb(y,\eta;x,\hb D_x)Q_{1,f}(y,\eta)^{1/2}\right)f(y,\eta)dyd\eta
\\
+(1+\tfrac1\a)\int\Tr_\fH\left(Q_{3,f}(y,\eta)^{1/2}c_\hb(y,\eta;z,\hb D_z)Q_{1,f}(y,\eta)^{1/2}\right)f(y,\eta)dyd\eta
\\
=(1+\a)\int\Tr_\fH\left((fQ_{1,f}(y,\eta))^{1/2}c_\hb(y,\eta;x,\hb D_x)(fQ_{1,f}(y,\eta))^{1/2}\right)dyd\eta
\\
+(1+\tfrac1\a)\int\Tr_\fH\left((fQ_{3,f}(y,\eta))^{1/2}c_\hb(y,\eta;z,\hb D_z)(fQ_{1,f}(y,\eta))^{1/2}\right)dyd\eta
\\
=(1+\a)\int\Tr_\fH\left(Q_1(y,\eta)^{1/2}c_\hb(y,\eta;x,\hb D_x)Q_1(y,\eta)^{1/2}\right)dyd\eta
\\
+(1+\tfrac1\a)\int\Tr_\fH\left(Q_3(y,\eta)^{1/2}c_\hb(y,\eta;z,\hb D_z)Q_3(y,\eta)^{1/2}\right)dyd\eta&\,.
\ea
$$
By construction
$$
P:=\int(Q_{1,f}\otimes Q_{3,f}(y,\eta))f((y,\eta)dyd\eta\in\cC(R_1,R_3)\,;
$$
on the other hand
$$
\ba
\int\Tr_\fH\left(Q_1(y,\eta)^{1/2}c_\hb(y,\eta;x,\hb D_x)Q_1(y,\eta)^{1/2}\right)dyd\eta<\infty
\\
\int\Tr_\fH\left(Q_3(y,\eta)^{1/2}c_\hb(y,\eta;z,\hb D_z)Q_3(y,\eta)^{1/2}\right)dyd\eta<\infty
\ea
$$
since $R_1,R_3\in\cD_2(\fH)$ while $f\in\cP_2^{ac}(\bR^d\times\bR^d)$. By Corollary \ref{C-Energ}
$$
\Tr_{\fH\otimes\fH}\left(PC_\hb^\eps(x,z,\hb D_x,\hb D_z)\right)\to\Tr_{\fH\otimes\fH}\left(P^{1/2}C_\hb(x,z,\hb D_x,\hb D_z)P^{1/2}\right)
$$
as $\eps\to 0^+$, so that
$$
\ba
MK_\hb(R_1,R_3)^2\le\Tr_{\fH\otimes\fH}\left(P^{1/2}C_\hb(x,z,\hb D_x,\hb D_z)P^{1/2}\right)
\\
\le(1+\a)\int\Tr_\fH\left(Q_1(y,\eta)^{1/2}c_\hb(y,\eta;x,\hb D_x)Q_1(y,\eta)^{1/2}\right)dyd\eta
\\
+(1+\tfrac1\a)\int\Tr_\fH\left(Q_3(y,\eta)^{1/2}c_\hb(y,\eta;z,\hb D_z)Q_3(y,\eta)^{1/2}\right)dyd\eta&\,.
\ea
$$
Minimizing the right hand side of this inequality in $Q_1\in\cC(f,R_1)$ and in $Q_\in\cC(f,R_3)$ shows that
$$
MK_\hb(R_1,R_3)^2\le(1+\a)\cE_\hb(f,R_1)^2+(1+\tfrac1\a)\cE_\hb(f,R_3)^2\,.
$$
Minimizing the right hand side of this inequality over $\a>0$, i.e. taking 
$$
\a=\cE_\hb(f,R_3)/\cE_\hb(f,R_1)
$$ 
(we recall that $\cE_\hb(f,R_1)\ge\sqrt{d\hb}>0$), we arrive at
$$
\ba
MK_\hb(R_1,R_3)^2\le&\cE_\hb(f,R_1)^2+\cE_\hb(f,R_3)^2+2\cE_\hb(f,R_1)\cE_\hb(f,R_3)
\\
=&(\cE_\hb(f,R_1)+\cE_\hb(f,R_3))^2\,,
\ea
$$
which is inequality (b).
\end{proof}

\smallskip
The proofs of inequalities (c)-(e) are simpler because of the rank-one assumption on the intermediate point $R_2$.

\begin{proof}[Proof of inequality (c)] According to Lemma \ref{L-rank1} (ii)
$$
\ba
MK_\hb(R_1,R_2)^2=\Tr_{\fH\otimes\fH}((R_1\otimes R_2)^{1/2}C(x,y,\hb D_x,\hb D_y)(R_1\otimes R_2)^{1/2})
\\
MK_\hb(R_2,R_3)^2=\Tr_{\fH\otimes\fH}((R_2\otimes R_3)^{1/2}C(y,z,\hb D_y,\hb D_z)(R_2\otimes R_3)^{1/2})
\ea
$$
since $R_2$ is a rank-one density. Applying the third inequality in Lemma \ref{L-IneqCost} shows that
$$
\ba
C^\eps_\hb(x,z,\hb D_x,\hb D_z)\le C_\hb(x,y,\hb D_x,\hb D_y)
\\
\le(1+\a)C_\hb(x,y,\hb D_x,\hb D_y)+(1+\tfrac1\a)C_\hb(y,z,\hb D_y,\hb D_z)
\ea
$$
so that
$$
\ba
(R_1\otimes R_2\otimes R_3)^{1/2}C^\eps_\hb(x,z,\hb D_x,\hb D_z)(R_1\otimes R_2\otimes R_3)^{1/2}
\\
\le(1+\a)(R_1\otimes R_2\otimes R_3)^{1/2}C_\hb(x,y,\hb D_x,\hb D_y)(R_1\otimes R_2\otimes R_3)^{1/2}
\\
+(1+\a)(R_1\otimes R_2\otimes R_3)^{1/2}C_\hb(y,z,\hb D_y,\hb D_z)(R_1\otimes R_2\otimes R_3)^{1/2}
\\
=(1+\a)\left((R_1\otimes R_2)^{1/2}C_\hb(x,y,\hb D_x,\hb D_y)(R_1\otimes R_2)^{1/2}\right)\otimes R_3
\\
+(1+\tfrac1\a)R_1\otimes\left((R_2\otimes R_3)^{1/2}C_\hb(y,z,\hb D_y,\hb D_z)(R_2\otimes R_3)^{1/2}\right)&\,.
\ea
$$
Taking the trace of both sides of this inequality in $\fH\otimes\fH\otimes\fH$
$$
\ba
\Tr_{\fH\otimes\fH}((R_1\otimes R_3)C^\eps_\hb(x,z,\hb D_x,\hb D_z))
\\
=\Tr_{\fH\otimes\fH\otimes\fH}((R_1\otimes R_2\otimes R_3)C^\eps_\hb(x,z,\hb D_x,\hb D_z))
\\
=\Tr_{\fH\otimes\fH\otimes\fH}((R_1\otimes R_2\otimes R_3)^{1/2}C^\eps_\hb(x,z,\hb D_x,\hb D_z)(R_1\otimes R_2\otimes R_3)^{1/2})
\\
\le(1+\a)\Tr_{\fH\otimes\fH}\left((R_1\otimes R_2)^{1/2}C_\hb(x,y,\hb D_x,\hb D_y)(R_1\otimes R_2)^{1/2}\right)
\\
+(1+\tfrac1\a)\Tr_{\fH\otimes\fH}\left((R_2\otimes R_3)^{1/2}C_\hb(y,z,\hb D_y,\hb D_z)(R_2\otimes R_3)^{1/2}\right)
\\
=(1+\a)MK_\hb(R_1,R_2)^2+(1+\tfrac1\a)MK_\hb(R_2,R_3)^2&\,.
\ea
$$
Passing to the limit as $\eps\to 0^+$ in the left hand side implies that
$$
\ba
MK_\hb(R_1,R_3)^2\le\Tr_{\fH\otimes\fH}((R_1\otimes R_3)^{1/2}C_\hb(x,z,\hb D_x,\hb D_z)(R_1\otimes R_3)^{1/2})
\\
=\lim_{\eps\to 0^+}\Tr_{\fH\otimes\fH}((R_1\otimes R_3)C^\eps_\hb(x,z,\hb D_x,\hb D_z))
\\
\le(1+\a)MK_\hb(R_1,R_2)^2+(1+\tfrac1\a)MK_\hb(R_2,R_3)^2&\,,
\ea
$$
where the first inequality follows from the definition of $MK_\hb$ and the fact that $R_1\otimes R_3\in\cC(R_1,R_3)$, and the equality from Corollary \ref{C-Energ}. 

Setting $\a:=MK_\hb(R_2,R_3)/MK_\hb(R_1,R_2)$, we arrive at
$$
\ba
MK_\hb(R_1,R_3)^2\le& MK_\hb(R_1,R_2)^2+MK_\hb(R_2,R_3)^2+2MK_\hb(R_1,R_2)MK_\hb(R_2,R_3)
\\
=&(MK_\hb(R_1,R_2)+MK_\hb(R_2,R_3))^2
\ea
$$
which is the inequality (c).
\end{proof}

\begin{proof}[Proof of inequality (d)] According to Lemma \ref{L-rank1} (i)
$$
\ba
\cE_\hb(f,R_2)^2=\int\Tr_{\fH}(R_2^{1/2}c_\hb(x,\xi)R_2^{1/2})f(x,\xi)dxd\xi
\\
\cE_\hb(g,R_2)^2=\int\Tr_{\fH}(R_2^{1/2}c_\hb(z,\zeta)R_2^{1/2})g(z,\zeta)dzd\zeta
\ea
$$
since $R_2$ is a rank-one density. Applying the fourth inequality in Lemma \ref{L-IneqCost} shows that
$$
\ba
|x-z|^2+|\xi-\zeta|^2\le(1+\a)c_\hb(x,\xi;y,\hb D_y)+(1+\tfrac1\a)c_\hb(z,\zeta;y,\hb D_y)\,,
\ea
$$
so that
$$
(|x-z|^2+|\xi-\zeta|^2)R_2\le(1+\a)R_2^{1/2}c_\hb(x,\xi)R_2^{1/2}+(1+\tfrac1\a)R_2^{1/2}c_\hb(z,\zeta)R_2^{1/2}
$$
for all $x,z,\xi,\zeta\in\bR^d$. Taking the trace of both sides of this inequality, and integrating in $x,\xi,z,\zeta$ after multiplying by $f(x,\xi)g(z,\zeta)$ shows that
$$
\ba
\MKd(f,g)^2\le\int(|x-z|^2+|\xi-\zeta|^2)f(x,\xi)g(z,\zeta)dxd\xi dz d\zeta
\\
=(1+\a)\int\Tr_\fH(R_2^{1/2}c_\hb(x,\xi)R_2^{1/2})f(x,\xi)dxd\xi
\\
+(1+\tfrac1\a)\int\Tr_\fH(R_2^{1/2}c_\hb(z,\zeta)R_2^{1/2})g(z,\zeta)dzd\zeta
\\
=(1+\a)\cE_\hb(f,R_2)^2+(1+\tfrac1\a)\cE_\hb(g,R_2)^2\,,
\ea
$$
since
$$
\Tr_\fH(R_2)=\int f(x,\xi)dxd\xi=\int g(z,\zeta)dzd\zeta=1\,.
$$
The first inequality comes from the definition of the Monge-Kantorovich-Wasserstein distance $\MKd$ and the fact that $f\otimes g$ is a (nonoptimal) coupling of $f$ and $g$. Choosing
$$
\a=\cE_\hb(g,R_2)/\cE_\hb(f,R_2)
$$
shows that
$$
\ba
\MKd(f,g)^2\le&\cE_\hb(f,R_2)^2+\cE_\hb(g,R_2)^2+2\cE_\hb(f,R_2)\cE_\hb(g,R_2)
\\
=&(\cE_\hb(f,R_2)+\cE_\hb(g,R_2))^2
\ea
$$
which is the inequality (d).
\end{proof}

\begin{proof}[Proof of inequality (e)] According to Lemma \ref{L-rank1} 
$$
\ba
\cE_\hb(f,R_2)^2=\int\Tr_{\fH}(R_2^{1/2}c_\hb(x,\xi;y,\hb D_y)R_2^{1/2})f(x,\xi)dxd\xi
\\
MK_\hb(R_2,R_3)^2=\Tr_{\fH\otimes\fH}((R_2\otimes R_3)^{1/2}C_\hb(y,z,\hb D_y,\hb D_z)(R_2\otimes R_3)^{1/2})
\ea
$$
since $R_2$ is a rank-one density. Applying the fifth inequality in Lemma \ref{L-IneqCost} shows that
$$
c_\hb(x,\xi;z,\hb D_z)\le(1+\a)c_\hb(x,\xi;y,\hb D_y)+(1+\tfrac1\a)C_\hb(y,z,\hb D_y,\hb D_z)
$$
so that, for each $\eps>0$
$$
0\le c^\eps_\hb(x,\xi;z,\hb D_z)\le(1+\a)c_\hb(x,\xi;y,\hb D_y)+(1+\tfrac1\a)C_\hb(y,z,\hb D_y,\hb D_z)
$$
with
$$
\ba
c^\eps_\hb(x,\xi;z,\hb D_z):=&(I+\eps c_\hb(x,\xi;z,\hb D_z))^{-1}c_\hb(x,\xi;z,\hb D_z)
\\
=&c^\eps_\hb(x,\xi;z,\hb D_z)^*\in\cL(\fH)
\ea
$$
for all $x,\xi\in\bR^d$. Hence
$$
\ba
R_2\otimes (R_3^{1/2}c^\eps_\hb(x,\xi;z,\hb D_z)R_3^{1/2})
\\
\le(1+\a)(R_2^{1/2}c_\hb(x,\xi;y,\hb D_y)R_2^{1/2})\otimes R_3
\\
+(1+\tfrac1\a)(R_2\otimes R_3)^{1/2}C_\hb(y,z,\hb D_y,\hb D_z)(R_2\otimes R_3)^{1/2}
\ea
$$
for all $x,\xi\in\bR^d$ and, taking the trace of both sides of this inequality leads to
\be\lb{Ineq-aexxi}
\ba
\Tr_\fH(R_3^{1/2}c^\eps_\hb(x,\xi;z,\hb D_z)R_3^{1/2})
\\
\le(1+\a)\Tr_\fH(R_2^{1/2}c_\hb(x,\xi;y,\hb D_y)R_2^{1/2})+(1+\tfrac1\a)MK_\hb(R_2,R_3)^2&\,.
\ea
\ee
Multiplying both sides of this inequality by $f(x,\xi)$ and integrating in $(x,\xi)$ shows that
$$
\ba
\int\Tr_\fH(R_3^{1/2}c^\eps_\hb(x,\xi;z,\hb D_z)R_3^{1/2})f(x,\xi)dxd\xi
\\
\le(1+\a)\cE_\hb(f,R_2)^2+(1+\tfrac1\a)MK_\hb(R_2,R_3)^2
\\
=(\cE_\hb(f,R_2)+MK_\hb(R_2,R_3))^2&\,,
\ea
$$
with the choice
$$
\a:=MK_\hb(R_2,R_3)/\cE_\hb(f,R_2)\,.
$$
Since the right-hand side of \eqref{Ineq-aexxi} is integrable with respect to $f(x,\xi)dxd\xi$, and therefore finite for $f(x,\xi)dxd\xi$-a.e. $(x,\xi)\in\bR^d\times\bR^d$, one has
$$
\Tr_\fH(R_3^{1/2}c^\eps_\hb(x,\xi;z,\hb D_z)R_3^{1/2})\to\Tr_\fH(R_3^{1/2}c\hb(x,\xi;z,\hb D_z)R_3^{1/2})
$$
for $f(x,\xi)dxd\xi$-a.e. $(x,\xi)\in\bR^d\times\bR^d$ by Corollary \ref{C-Energ}. By Fatou's lemma, observing that $f\otimes R_3\in\cC(f,R_3)$, one has
$$
\ba
\cE_\hb(f,R_3)\le\int\Tr_\fH(R_3^{1/2}c_\hb(x,\xi;z,\hb D_z)R_3^{1/2})f(x,\xi)dxd\xi
\\
\le\varliminf_{\eps\to 0^+}\int\Tr_\fH(R_3^{1/2}c^\eps_\hb(x,\xi;z,\hb D_z)R_3^{1/2})f(x,\xi)dxd\xi
\\
\le(\cE_\hb(f,R_2)+MK_\hb(R_2,R_3))^2&\,,
\ea
$$
which is the inequality (e). 
\end{proof}

\section{Applications}
%%%%%%%%%%%%%%%%%%%%%%%%%%%%%%%%%%%%%%%%%%%%%%%%%%%%%%%%%%%%%%%%%%%%%%%%%%%%%%%%%%%%%%%%%%%%%%%%%%%%%%%%%%%%%%%%%%%%%%%

One satisfying consequence of the triangle inequalities proved in the last section is the following statement, which confirms that $MK_\hb$ can indeed be thought of as a quantum deformation of the quadratic Monge-Kantorovich-Wasserstein 
distance.

\begin{Thm}
Let $R_\hb,S_\hb$ be families of density operators in $\cD_2(\fH)$, and let $f,g\in\cP^{ac}_2(\bR^d\times\bR^d)$. Assume that 
$$
\cE_\hb(f,R_\hb)\to 0\quad\text{ and }\quad\cE_\hb(g,S_\hb)\to 0
$$
as $\hb\to 0$. Then
$$
\lim_{\hb\to 0}MK_\hb(R_\hb,S_\hb)=\MKd(f,g)\,.
$$
\end{Thm}

\smallskip
This statement is to be compared with the lower bound
$$
MK_\hb(R_\hb,S_\hb)^2\ge\MKd(\tilde W_\hb(R_\hb),\tilde W_\hb(S_\hb))^2-2d\hb\,,
$$
which is Theorem 2.3 (2) in \cite{FGMouPaul}, and with the upper bound obtained in the special case of T\"oplitz operators
$$
MK_\hb(\Op^T_\hb((2\pi\hb)^d\mu),\Op^T_\hb((2\pi\hb)^d\nu))^2\le\MKd(\mu,\nu)^2+2d\hb\,,
$$
stated as Theorem 2.3 (1) in \cite{FGMouPaul}.

\begin{proof}
By Theorem \ref{T-IneqT} (a)-(b), 
$$
\ba
MK_\hb(R_\hb,S_\hb)\le&\cE_\hb(f,R_\hb)+\cE_\hb(f,S_\hb)
\\
\le&\cE_\hb(f,R_\hb)+\MKd(f,g)+\cE_\hb(g,S_\hb)\,.
\ea
$$
Hence
$$
\varlimsup_{\hb\to 0^+}MK_\hb(R_\hb,S_\hb)\le\MKd(f,g)\,.
$$
By Theorem 2.4 (2) in \cite{FGPaul}
$$
\MKd(f,\tilde W(R_\hb))^2\le\cE_\hb(f,R_\hb)^2+d\hb
$$
(notice the slight change of normalization in the definition of $\cE_\hb$ between \cite{FGPaul} and the present paper), so that our assumption implies that
$$
\MKd(f,\tilde W(R_\hb))\to 0\quad\text{ and }\quad\MKd(g,\tilde W(S_\hb))\to 0
$$
as $\hb\to 0$. From the inequality 
$$
\MKd(\tilde W_\hb(R_\hb),\tilde W_\hb(S_\hb))^2\le MK_\hb(R_\hb,S_\hb)^2+2d\hb\,,
$$
(Theorem 2.3 (2) in \cite{FGMouPaul}), we deduce that
$$
\MKd(f,g)\le\varliminf_{\hb\to 0}MK_\hb(R_\hb,S_\hb)\,.
$$
Notice that this last lower bound is a variant of the last inequality in Theorem 2.3 of \cite{FGMouPaul}, except that in the present case the assumption on $R_\hb$ and $S_\hb$ is slightly different (in other words, we have assumed that
$\cE_\hb(f,R_\hb)\to 0$ instead of assuming that $\tilde W_\hb(R_\hb)\to f$ in $\cS'(\bR^d\times\bR^d)$.)
\end{proof}

%%%%%%%%%%%%%%%%%%%%%%%%%%%%%%%%%%%%%%%%%%%%%%%%%%%%%%%%%%%%%%%%%%%%%%%%%%%%%%%%%%%%%%%%%%%%%%%%%%%%%%%%%%%%%%%%%%%%%%%
\section{Kantorovich duality for $\cE_\hb$}\label{kantodual}
%%%%%%%%%%%%%%%%%%%%%%%%%%%%%%%%%%%%%%%%%%%%%%%%%%%%%%%%%%%%%%%%%%%%%%%%%%%%%%%%%%%%%%%%%%%%%%%%%%%%%%%%%%%%%%%%%%%%%%%
%\subsection{Duality}\label{dual}
\begin{Thm}\lb{T-DualityARMA}
Let $S\in\cD_2(\fH)$ and let $p\equiv p(x,\xi)$ be a probability density on $\bR^{2d}$ such that
$$
\int_{\bR^{2d}}(|x|^2+|\xi|^2)p(x,\xi)dxd\xi<+\infty\,.
$$
Then
$$
\ba
\cE_\hb(p,S)^2=&\min_{Q\in\cC(p,S)}\int_{\bR^{2d}}\Tr_\fH(Q(x,\xi)^{1/2}c(x,\xi)Q(x,\xi)^{1/2})dxd\xi
\\
=&\sup_{a\in C_b(\bR^{2d}),\,B=B^*\in\cL(\fH)\atop a(x,\xi)I_\fH+B\le c(x,\xi)}\left(\int_{\bR^{2d}}a(x,\xi)p(x,\xi)dxd\xi+\Tr_\fH(BS)\right)\,.
\ea
$$
\end{Thm}

Notice that the duality theorem implies in particular the existence of at least one optimal coupling $Q\in\cC(p,S)$.

\begin{proof} The proof is split in several steps.

\smallskip
\noindent
\textit{Step 1: the functions $f$ and $g$.} Consider the Banach space $E:=C_b(\bR^{2d};\cL(\fH))$, with
$$
\|T\|_E:=\sup_{(x,\xi)\in\bR^{2d}}\|T(x,\xi)\|\,,
$$
and set
$$
f(T):=\left\{\begin{aligned}{}&0\quad&&\text{ if }T(x,\xi)=T(x,\xi)^*\ge -c(x,\xi)\text{ for all }(x,\xi)\in\bR^{2d}\,,\\ &+\infty&&\text{ otherwise,}\end{aligned}\right.
$$
while
$$
g(T):=\left\{\begin{aligned}{}&\int_{\bR^{2d}}ap(x,\xi)dxd\xi\!+\!\Tr_\fH(BS)\quad&&\text{ if }T(x,\xi)\!=\!T(x,\xi)^*\!=\!a(x,\xi)I_\fH\!+\!B
\\ 
& &&\text{ for all }(x,\xi)\in\bR^{2d}\,,
\\ 
&+\infty&&\text{ otherwise,}\end{aligned}\right.
$$
The constraint $T(x,\xi)=T(x,\xi)^*\ge -c(x,\xi)$ means that, for each $(x,\xi)\in\bR^{2d}$, one has
$$
\la\phi(x,\xi)|T(x,\xi)+c(x,\xi)|\phi(x,\xi)\ra\ge 0
$$
for all $\phi\in\fDom(c(x,\xi))$. On the other hand, the nullspace of the linear map
$$
C_b(\bR^{2d})\times\cL(\fH)\ni(a,B)\mapsto\Ga(a,B)\equiv a(x,\xi)I_\fH+B\in E
$$
is
$$
\Ker(\cL)=\{(t,-tI_\fH)\,,\quad t\in\bR\}\,.
$$
Since
$$
g((a\!+\!t)I_\fH\!+\!(B\!-\!tI_\fH))\!=\!g(aI_\fH\!+\!B)\!+\!t\int_{\bR^{2d}}p(x,\xi)dxd\xi\!-\!t\Tr_\fH(S)\!=\!g(aI_\fH\!+\!B)\,,
$$
the prescription above defines $g$ on $\Img(\Ga)\simeq(C_b(\bR^{2d})\times\cL(\fH))/\Ker(\Ga)$. Observe that
$$
\ba
g((aI_\fH+B)^*)=g(\bar a I_\fH+B^*)=&\int_{\bR^{2d}}\overline{a(x,\xi)}p(x,\xi)dxd\xi+\Tr_\fH(B^*S)
\\
=&\overline{\int_{\bR^{2d}}a(x,\xi)p(x,\xi)dxd\xi}+\Tr_\fH((SB)^*)
\\
=&\overline{\int_{\bR^{2d}}a(x,\xi)p(x,\xi)dxd\xi}+\overline{\Tr_\fH(SB)}
\\
=&\overline{g(aI_\fH+B)}\,,
\ea
$$
so that $(aI_\fH+B)^*=aI_\fH+B\implies g(aI_\fH+B)\in\bR$. Thus the definition above implies that $g$ takes its values in $(-\infty,+\infty]$.

The functions $f$ and $g$ are convex. Indeed, $g$ is the extension by $+\infty$ of a $\bR$-linear functional defined on the set of self-adjoint elements of $\Img(\Ga)$, which is a linear subspace of $E$. As for $f$, it is the indicator function
(in the sense of the definition in \S 4 of \cite{Rockafellar} on p. 28) of the convex set
$$
\{T\in E\text{ s.t. }T(x,\xi)=T(x,\xi)^*\ge -c(x,\xi)\text{ for all }(x,\xi)\in\bR^{2d}\}
$$
and is therefore convex. Besides $f(0)=g(0)=0$, and $f$ is continuous at $0$. Indeed, by the Heisenberg inequality
$$
c(x,\xi)\ge d\hb I_\fH\quad\text{ for all }(x,\xi)\in\bR^{2d}\,,
$$
so that, for each $T\in E$
$$
\ba
T(x,\xi)=T(x,\xi)^*\text{ and }\|T(x,\xi)\|<\tfrac12d\hb\text{ for all }(x,\xi)\in\bR^{2d}
\\
\implies T(x,\xi)\ge-c(x,\xi)\text{ for all }(x,\xi)\in\bR^{2d}\implies f(T)=0&\,.
\ea
$$
In particular $f$ is continuous at $0$.

\smallskip
\noindent
\textit{Step 2: applying convex duality.} By the Fenchel-Rockafellar convex duality theorem (Theorem 1.12 in \cite{Brezis})
$$
\inf_{T\in E}(f(T)+g(T))=\max_{\L\in E'}(-f^*(-\L)-g^*(\L))\,.
$$
Let us compute the Legendre duals $f^*$ and $g^*$. 

First
$$
f^*(-\L)=\sup_{T\in E}(\la -\L,T\ra-f(T))=\sup_{T\in E\atop T(x,\xi)=T(x,\xi)^*\ge-c(x,\xi)}\la-\L,T\ra\,.
$$
If $\L\in E'$ is not a nonnegative linear functional, there exists $T_0\in E$ such that $T_0(x,\xi)=T_0(x,\xi)^*\ge 0$ such that $\la\L,T_0\ra=-\a<0$. Since 
$$
nT_0(x,\xi)=nT_0(x,\xi)^*\ge 0\ge-d\hb I_\fH\ge-c(x,\xi)\text{ for all }(x,\xi)\in\bR^{2d}
$$
one has
$$
f^*(-\L)\ge\sup_{n\ge 1}\la-\L,nT_0\ra=\sup_{n\ge 1}(n\a)=+\infty\,.
$$
For $\L\in E'$ such that $\L\ge 0$, we define
$$
\la\L,c\ra:=\sup_{T\in E\atop T(x,\xi)=T(x,\xi)^*\le c(x,\xi)}\la\L,T\ra\in[0,+\infty]\,.
$$
(Observe indeed that $T=0$ satisfies the constraints since $c(x,\xi)=c(x,\xi)^*\ge 0$ for each $(x,\xi)\in\bR^{2d}$.) With this definition, one has clearly
$$
f^*(-\L):=\left\{\ba{}&\la\L,c\ra\quad&&\text{ if }\L\ge 0\,,\\&0\quad&&\text{ otherwise.}\ea\right.
$$

Next
$$
\ba
g^*(\L)=&\sup_{T\in E}(\la\L,T\ra-g(T))
\\
=&\sup_{T\in E\atop T(x,\xi)=T(x,\xi)^*=a(x,\xi)I_\fH+B}\left(\la\L,T\ra-\int_{\bR^{2d}}a(x,\xi)p(x,\xi)dxd\xi-\Tr_\fH(BS)\right)\,.
\ea
$$
If there exists $a\equiv a(x,\xi)\in C_b(\bR^{2d},\bR)$ and $B=B^*\in\cL(\fH)$ such that either
$$
\la\L,aI_\fH+B\ra>\int_{\bR^{2d}}a(x,\xi)p(x,\xi)dxd\xi+\Tr_\fH(BS)
$$
or
$$
\la\L,aI_\fH+B\ra<\int_{\bR^{2d}}a(x,\xi)p(x,\xi)dxd\xi+\Tr_\fH(BS)\,,
$$
one has either
$$
g(\L)\ge\sup_{n\ge 1}\left(\la\L,n(aI_\fH+B)\ra-n\int_{\bR^{2d}}a(x,\xi)p(x,\xi)dxd\xi-n\Tr_\fH(BS)\right)=+\infty\,,
$$
or
$$
g(\L)\ge\sup_{n\ge 1}\left(\la\L,n(-aI_\fH-B)\ra+n\int_{\bR^{2d}}a(x,\xi)p(x,\xi)dxd\xi+n\Tr_\fH(BS)\right)=+\infty\,.
$$
Hence
$$
g^*(\L):=\left\{\ba{}&0\quad&&\text{ if }\la\L,aI_\fH+B\ra=\int_{\bR^{2d}}a(x,\xi)p(x,\xi)dxd\xi+\Tr_\fH(BS)
\\& &&\text{ for each }a\equiv a(x,\xi)\in C_b(\bR^{2d},\bR)\text{ and }B=B^*\in\cL(\fH)\,,
\\&+\infty\quad&&\text{ otherwise.}\ea\right.
$$
Notice that the prescription
$$
\la\L,aI_\fH+B\ra=\int_{\bR^{2d}}a(x,\xi)p(x,\xi)dxd\xi+\Tr_\fH(BS)
$$
defines a unique linear functional on the set of $T\in\Img\Ga$ such that $T(x,\xi)^*=T(x,\xi)$ for each $(x,\xi)\in\bR^{2d}$ by the same argument as in Step 1. 

Therefore, the Fenchel-Rockafellar duality theorem in this case results in the equality
$$
\ba
\inf_{T\in E}(f(T)+g(T))=\inf_{a\in C_b(\bR^{2d},\bR),\,B=B^*\atop a(x,\xi)I_\fH+B\ge-c(x,\xi)}\left(\int_{\bR^{2d}}a(x,\xi)p(x,\xi)dxd\xi+\Tr_\fH(SB)\right)
\\
=\max_{\L\in E'}(f^*(-\L)+g^*(\L))=\max_{0\le\L\in E',\,\,\,\la\L,aI\fH+B\ra\atop =\int a(x,\xi)p(x,\xi)dxd\xi+\Tr(SB)}-\la\L,c\ra
\ea
$$
or, equivalently
$$
\ba
\sup_{a\in C_b(\bR^{2d},\bR),\,B=B^*\atop a(x,\xi)I_\fH+B\le c(x,\xi)}\left(\int_{\bR^{2d}}a(x,\xi)p(x,\xi)dxd\xi+\Tr_\fH(SB)\right)
\\
=\min_{0\le\L\in E',\,\,\,\la\L,aI_\fH+B\ra\atop =\int a(x,\xi)p(x,\xi)dxd\xi+\Tr(SB)}\la\L,c\ra&\,.
\ea
$$

\smallskip
\noindent
\textit{Step 3: representing the optimal $\L$.} Define a linear map  $F_\L:\, C_b(\bR^{2d})\to\cL^1(\fH)$ by the formula
$$
\Tr_\fH(KF_\L(a))=\L(aK)\,,\qquad\text{ for each }K\in\cK(\fH)\,.
$$
Indeed, since $K\mapsto\L(aK)$ is a linear functional on $\cK(\fH)$ which is continuous for the norm topology, and since $\cK(\fH)'=\cL^1(\fH)$, this linear functional is represented by a trace-class operator $F_\L(a)$. Since $\L$
is linear, the map $F_\L$ is linear.

 Since $\L\ge 0$, one has $F_\L(a)=F_\L(a)^*\ge 0$ for each $a\in C_b(\bR^d)$ such that $a(x,\xi)\ge 0$ for each $(x,\xi)\in\bR^{2d}$. Indeed, for $a\in C_b(\bR^{2d};\bR)$, set 
 $$
 T_1:=\tfrac12(F_\L(a)+F_\L(a)^*)\,,\qquad T_2:=-\tfrac12i(F_\L(a)-F_\L(a)^*)\,.
 $$
 Then, for each $K=K^*\in\cK(\fH)$, one has
 $$
 \L(aK)=\Tr_\fH(T_1K)+i\Tr_\fH(T_2K)
 $$
 with
 $$
 \overline{Tr_\fH(T_jK)}=Tr_\fH((T_jK)^*)=Tr_\fH(K^*T_j^*)=Tr_\fH(KT_j)=Tr_\fH(T_jK)
 $$
 for $j=1,2$. Since $a\in C_b(\bR^{2d};\bR)$ and $K=K^*\in\cL(\fH)$, one has
 $$
 -\|a\|_{L^\infty}\|K\|I_\fH\le aK\le\|a\|_{L^\infty}\|K\|I_\fH
 $$
 so that 
 $$
-\|a\|_{L^\infty}\|K\|\le\L(aK)\le\|a\|_{L^\infty}\|K\|\quad\text{ since }\L(I_\fH)=\int_{\bR^{2d}}p(x,\xi)dxd\xi=1\,.
$$
In particular, $\L(aK)\in\bR$, so that $\Tr_\fH(T_2K)=0$ for each $K=K^*\in\cK(\fH)$. Since $T_2=T_2^*\in\cL^1(\fH)$, specializing this identity to the case where $K$ is the orthogonal projection on any eigenvector of $T_2$
shows that $T_2=0$. Thus 
$$
a\in C_b(\bR^{2d};\bR)\implies F_\L(a)=F_\L(a)^*\,.
$$
Moreover
$$
a\in C_b(\bR^{2d};\bR)\text{ and }a\ge 0\implies\Tr_\fH(F_\L(a)K)\ge 0\text{ for each }K=K^*\ge 0\text{ in }\cK(\fH)
$$
and specializing this last inequality to the case where $K$ is the orthogonal projection on any eigenvector of $F_\L(a)=F_\L(a)^*\in\cL^1(\fH)$ shows that all the eigenvalues of $F_\L(a)$ are nonnegative, so that $F_\L(a)\ge 0$.

Next we deduce from the defining identity for $F_\L$, i.e.
$$
\L(aK)=\Tr_\fH(F_\L(a)K)\text{ for each }a\in C_b(\bR^{2d};\bC)\text{ and }K\in\cK(\fH)
$$
that 
$$
\|F_\L(a)\|_1\le\|\L\|\|a\|_{L^\infty(\bR^{2d})}\,.
$$
Next we specialize this defining identity to the case where $a\ge 0$ on $\bR^{2d}$ while $K=\Pi_n$ is the orthogonal projection on $\Span\{e_1,\ldots,e_n\}$, with $(e_1,e_2,\ldots)$ a complete orthonormal system in $\fH$.
One has
$$
\L(a\Pi_n)=\Tr_\fH(F_\L(a)\Pi_n)\to\Tr_\fH(F_\L(a))=\|F_\L(a)\|_1\quad\text{ as }n\to\infty
$$
while
$$
a(I_\fH-\Pi_n)\ge 0\text{ so that }\L(a\Pi_n)\le\L(aI_\fH)=\int_{\bR^{2d}}a(x,\xi)p(x,\xi)dxd\xi
$$
so that
$$
a\in C_b(\bR^{2d})\text{ and }a\ge 0\implies\|F_\L(a)\|_1\le\int_{\bR^{2d}}a(x,\xi)p(x,\xi)dxd\xi\,.
$$
More generally, for each $a\in C_b(\bR^{2d};\bR)$, one has $-|a|\le a\le|a|$ so that 
$$
|\Tr_\fH(F_\L(a)|e_j\ra\la e_j|)|=|\L(a|e_j\ra\la e_j|)|\le\L(|a||e_j\ra\la e_j|)
$$
for each $j\ge 1$, where $(e_1,e_2,\ldots,)$ is a complete orthonormal system of eigenvectors of $F_\L(a)=F_\L(a)^*\in\cL^1(\fH)$. Hence
$$
\sum_{j=1}^n|\Tr_\fH(F_\L(a)|e_j\ra\la e_j|)|\le\L\left(|a|\sum_{j=1}^n|e_j\ra\la e_j|\right)\le\L(|a|I_\fH)\,,
$$
and since
$$
\sum_{j=1}^n|\Tr_\fH(F_\L(a)|e_j\ra\la e_j|)|\to\|F_\L(a)\|_1\quad\text{ as }n\to\infty
$$
we conclude that
$$
\|F_\L(a)\|_1\le\L(|a|I_\fH)=\int_{\bR^{2d}}|a(x,\xi)|p(x,\xi)dxd\xi\,.
$$
Since $C_b(\bR^{2d})$ is dense in $L^1(\bR^{2d},pdxd\xi)$, this inequality, applied to the real and the imaginary part of $a$, shows that $F_\L$ is a continuous linear operator from $L^1(\bR^{2d}$ to $\cL^1(\fH)$. Since $\cL^1(\fH)$ 
is separable and is the dual of the Banach space $\cK(\fH)$ (the norm closure in $\cL(\fH)$ of the set of finite rank operators), we conclude from the Dunford-Pettis theorem (Theorem 1 in \S 3 of chapter III in \cite{DiestelUhl}) that
$\cL^1(\fH)$ has the Radon-Nikodym property. By Theorem 5 in \S 1 of chapter III in \cite{DiestelUhl}, the operator $F_\L$ is Riesz-representable: in other words, there exists $q\in L^\infty(\bR^{2d},pdxd\xi;\cL^1(\fH))$ such that
 $$
 F_\L(a)=\int_{\bR^{2d}}a(x,\xi)q(x,\xi)p(x,\xi)dxd\xi\,,\quad\text{ for all }a\in L^1(\bR^{2d},pdxd\xi)\,.
 $$

\smallskip
\noindent
\textit{Step 4: defining the optimal coupling.} We have seen that 
$$
a\in C_b(\bR^{2d})\text{ and }a\ge 0\implies F_\L(a)=\int_{\bR^{2d}}a(x,\xi)q(x,\xi)p(x,\xi)dxd\xi\ge 0\,.
$$
This implies that $q(x,\xi)=q(x,\xi)^*\ge 0$ for a.e. $(x,\xi)\in\bR^{2d}$.

Next, one has
$$
\L(K)=\Tr_\fH(F_\L(1)K)=\Tr_\fH(KS)\,,\qquad K\in\cK(\fH)\,,
$$
so that 
$$
F_\L(1)=\int_{\bR^{2d}}q(x,\xi)p(x,\xi)dxd\xi=S\in\cL^1(\fH)=\cK(\fH)'\,.
$$

On the other hand, for each $a\in C_b(\bR^{2d})$ such that $a\ge 0$, one has
$$
\ba
\Tr_\fH\left(P_n\int_{\bR^{2d}}a(x,\xi)q(x,\xi)p(x,\xi)dxd\xi\right)=\Tr_\fH(F_\L(a)P_n)=\L(aP_n)
\\
\le\L(aI_\fH)=\int_{\bR^{2d}}a(x,\xi)d(x,\xi)dxd\xi
\ea
$$
where $P_n$ is the orthogonal projection on $\Span\{e_1,\ldots,e_n\}$, with $(e_1,e_2,\ldots)$ being a complete orthonormal system of eigenvectors of
$$
\int_{\bR^{2d}}a(x,\xi)q(x,\xi)p(x,\xi)dxd\xi\in\cL^1(\fH)\,.
$$
Letting $n\to\infty$, one has
$$
\ba
\Tr_\fH\left(P_n\int_{\bR^{2d}}a(x,\xi)q(x,\xi)p(x,\xi)dxd\xi\right)
\\
\to\Tr_\fH\left(\int_{\bR^{2d}}a(x,\xi)q(x,\xi)p(x,\xi)dxd\xi\right)&\,,
\ea
$$
so that
$$
\int_{\bR^{2d}}a(x,\xi)\Tr_\fH(q(x,\xi))p(x,\xi)dxd\xi\le\int_{\bR^{2d}}a(x,\xi)d(x,\xi)dxd\xi\,.
$$
Since this holds for each $a\in C_b(\bR^{2d}$ such that $a\ge 0$, we conclude that
$$
\Tr_\fH(q(x,\xi))\le 1\quad\text{ for $p(x,\xi)dxd\xi$--a.e. }(x,\xi)\in\bR^{2d}\,.
$$
Moreover
$$
\int_{\bR^{2d}}(1-\Tr_\fH(q(x,\xi)))p(x,\xi)dxd\xi=1-\Tr_\fH(S)=0
$$
so that
$$
\Tr_\fH(q(x,\xi))=1\quad\text{ for $p(x,\xi)dxd\xi$--a.e. }(x,\xi)\in\bR^{2d}\,.
$$
In other words, we have proved that $(x,\xi)\mapsto Q(x,\xi)=p(x,\xi)q(x,\xi)$ defines an element of $\cC(p,S)$.

\smallskip
\noindent
\textit{Step 5: extending the representation formula for $\L$.} For each $B\in E$, we define
$$
\la L,B\ra:=\la\L,B\ra-\int_{\bR^{2d}}\Tr_\fH(B(x,\xi)Q(x,\xi))dxd\xi\,.
$$
Let us prove that
$$
B\in E\text{ and }B(x,\xi)=B(x,\xi)^*\ge 0\text{ for all }(x,\xi)\in\bR^{2d}\implies \la L,B\ra\ge 0\,.
$$
Pick $\eps>0$, and let $Q_\eps$ be a simple $\cL^1(\fH)$-valued function on $\bR^{2d}$ such that
$$
\int_{\bR^{2d}}\|Q(x,\xi)-Q_\eps(x,\xi)\|_1dxd\xi<\eps\,.
$$
Write
$$
Q_\eps(x,\xi)=\sum_{j=1}^N\indc_{\Om_j}(x,\xi)Q_j\,,\quad 0\le Q_j=Q_j^*\in\cL^1(\fH)\text{ for each }j=0,\ldots,N\,,
$$
where $\Om_j$ are bounded, pairwise disjoint measurable sets in $\bR^{2d}$ for $j=1,\ldots,N$. For each $j=1,\ldots,N$, let $(e_{j,1},e_{j,2},\ldots)$ designate a complete orthonormal system of eigenvectors of $Q_j$, and let $P_{j,n}$ 
be the orthogonal projection on $\Span\{e_{j,1},\ldots,e_{j,n}\}$. Define 
$$
\Pi_n(x,\xi)=\sum_{j=1}^N\indc_{\Om_j}(x,\xi)P_{j,n}\,.
$$
One easily checks that $\Pi_n(x,\xi)=\Pi_n(x,\xi)^*=\Pi_n(x,\xi)^2$ for each $(x,\xi)\in\bR^{2d}$. Then, for each $B\in E$ such that $B(x,\xi)=B(x,\xi)^*\ge 0$ for all $(x,\xi)\in\bR^{2d}$, one has
$$
\ba
0\le&\la\L,(I_\fH-\Pi_n)B(I_\fH-\Pi_n)\ra=\la\L,B\ra-\la\L,\Pi_nB+B\Pi_n-\Pi_nB\Pi_n\ra
\\
=&\la\L,B\ra-\int_{\bR^{2d}}\Tr_\fH((\Pi_nB+B\Pi_n-\Pi_nB\Pi_n)Q)(x,\xi)dxd\xi
\\
=&\la\L,B\ra-\int_{\bR^{2d}}\Tr_\fH((\Pi_nB+B\Pi_n-\Pi_nB\Pi_n)Q_\eps)(x,\xi)dxd\xi
\\
&+\int_{\bR^{2d}}\Tr_\fH((\Pi_nB+B\Pi_n-\Pi_nB\Pi_n)(Q_\eps-Q))(x,\xi)dxd\xi\,.
\ea
$$
By construction, keeping $\eps>0$ fixed, one has
$$
\ba
\int_{\bR^{2d}}\Tr_\fH((\Pi_nB+B\Pi_n-\Pi_nB\Pi_n)Q_\eps)(x,\xi)dxd\xi&
\\
=\int_{\bR^{2d}}\Tr_\fH(B\Pi_nQ_\eps\Pi_n)(x,\xi)dxd\xi&\to\int_{\bR^{2d}}\Tr_\fH(BQ_\eps)(x,\xi)dxd\xi
\ea
$$
as $n\to\infty$, so that
$$
\ba 
0\le\varlimsup_{n\to\infty}\la\L,(I_\fH-\Pi_n)B(I_\fH-\Pi_n)\ra=\la\L,B\ra-\int_{\bR^{2d}}\Tr_\fH(BQ)(x,\xi)dxd\xi
\\
+\varlimsup_{n\to\infty}\int_{\bR^{2d}}\Tr_\fH((\Pi_nB+B\Pi_n-\Pi_nB\Pi_n)(Q_\eps-Q))(x,\xi)dxd\xi
\\
+\int_{\bR^{2d}}\Tr_\fH(B(Q-Q_\eps))(x,\xi)dxd\xi&\,.
\ea
$$
On the other hand
$$
\ba
\left|\int_{\bR^{2d}}\Tr_\fH((\Pi_nB+B\Pi_n-\Pi_nB\Pi_n)(Q_\eps-Q))(x,\xi)dxd\xi\right|
\\
\le\int_{\bR^{2d}}|\Tr_\fH((\Pi_nB+B\Pi_n-\Pi_nB\Pi_n)(Q_\eps-Q))(x,\xi)|dxd\xi
\\
\le\int_{\bR^{2d}}\|(\Pi_nB+B\Pi_n-\Pi_nB\Pi_n)(x,\xi)\|\|(Q_\eps-Q)(x,\xi)\|_1dxd\xi
\\
\le 3\sup_{(x,\xi)\in\bR^{2d}}\|B(x,\xi)\|\int_{\bR^{2d}}\|(Q_\eps-Q)(x,\xi)\|_1dxd\xi
\\
\le 3\eps\sup_{(x,\xi)\in\bR^{2d}}\|B(x,\xi)\|\ea
$$
while, by the same token,
$$
\left|\int_{\bR^{2d}}\Tr_\fH(B(Q-Q_\eps))(x,\xi)dxd\xi\right|\le\eps\sup_{(x,\xi)\in\bR^{2d}}\|B(x,\xi)\|\,.
$$
Finally
$$
\la\L,B\ra-\int_{\bR^{2d}}\Tr_\fH(BQ)(x,\xi)dxd\xi\ge -4\eps\sup_{(x,\xi)\in\bR^{2d}}\|B(x,\xi)\|
$$
and since this holds for each $\eps>0$, we conclude that
$$
B\in E\text{ and }B(x,\xi)=B(x,\xi)^*\ge 0\text{ for all }(x,\xi)\in\bR^{2d}\implies \la L,B\ra\ge 0\,.
$$
By a classical argument, this implies that $\|L\|=\la L,I_\fH\ra$. 

On the other hand
$$
\ba
\la L,I_\fH\ra=\la L,I_\fH\ra-\int_{\bR^{2d}}\Tr_\fH(q(x,\xi))p(x,\xi)dxd\xi
\\
=\Tr_\fH(S)-\int_{\bR^{2d}}p(x,\xi)dxd\xi=0
\ea
$$
so that $L=0$. In other words, the representation formula
$$
\la\L,B\ra=\int_{\bR^{2d}}\Tr_\fH(B(x,\xi)Q(x,\xi))dxd\xi
$$
holds for each $B\in E$, and not only for $B\in C_b(\bR^{2d};\cK(\fH))$.

\smallskip
\noindent
\textit{Step 6: computing $\la\L,c\ra$.} As explained in Step 2
$$
\la\L,c\ra=\sup_{T\in E\atop T(x,\xi)=T(x,\xi)^*\le c(x,\xi)}\la\L,T\ra\,.
$$
For each $n\ge 1$, set
$$
c_n(x,\xi):=(I_\fH+\tfrac1nc(x,\xi))^{-1}c(x,\xi)\in\cL(\fH)\,,
$$
so that 
$$
0\le c_1(x,\xi)=c_1(x,\xi)^*\le\ldots\le c_n(x,\xi)=c_n(x,\xi)^*\le\ldots\le c(x,\xi)=c(x,\xi)^*\,.
$$
Thus, by definition
$$
\la\L,c_n\ra=\int_{\bR^{2d}}\Tr_\fH(Q(x,\xi)c_n(x,\xi))dxd\xi\le\la\L,c\ra
$$
for each $n\ge 1$, so that, by Corollary \ref{C-Energ}
$$
\ba
\int_{\bR^{2d}}\Tr_\fH(Q(x,\xi)^{1/2}c(x,\xi)Q(x,\xi)^{1/2})dxd\xi
\\
=\lim_{n\to\infty}\int_{\bR^{2d}}\Tr_\fH(Q(x,\xi)c_n(x,\xi))dxd\xi
\\
\le\la\L,c\ra&\,.
\ea
$$

On the other hand, let $(e_1(x,\xi),e_2(x,\xi),\ldots,)$ designate a complete orthonormal system in $\fH$ of eigenfunctions of $c(x,\xi)$, with $c(x,\xi)e_j(x,\xi)=\l_je_j(x,\xi)$ for $j\ge 1$. Since $c(x,\xi)$ is a phase space translate of 
the harmonic oscillator $H:=\tfrac12(|x|^2-\hb^2\Dlt_x)$, the eigenvalues $\l_j$ are independent of $(x,\xi)$. Set
$$
t_{kl}(x,\xi):=\la e_k(x,\xi)|Q(x,\xi)^{1/2}|e_l(x,\xi)\ra\,,\quad k,l\ge 1\,.
$$
Since $(x,\xi)\mapsto Q(x,\xi)^{1/2}\in L^2(\bR^{2d};\cL^2(\fH))$, one has
$$
v_k(x,\xi):=\sum_{l\ge 1}t_{kl}(x,\xi)e_l(x,\xi)\in\fDom(c(x,\xi))\quad\text{ for a.e. }(x,\xi)\in\bR^{2d}
$$
and
$$
\ba
\Tr_\fH(Q(x,\xi)^{1/2}c(x,\xi)Q(x,\xi)^{1/2})=&\sum_{k,l\ge 1}\l_l|t_{kl}(x,\xi)|^2
\\
=&\sum_{k\ge 1}\la v_k(x,\xi)|c(x,\xi)|v_k(x,\xi)\ra<\infty
\ea
$$
for a.e. $(x,\xi)\in\bR^{2d}$, since
$$
\int_{\bR^{2d}}\Tr_\fH(Q(x,\xi)^{1/2}c(x,\xi)Q(x,\xi)^{1/2})dxd\xi<\infty\,.
$$
Taking this last inequality for granted, we conclude as follows. Let $a\equiv a(x,\xi)\in C_b(\bR^{2d})$ and $B=B^*\in\cL(\fH)$ satisfy the constraint
$$
a(x,\xi)I_\fH+B\le c(x,\xi)\,,\quad(x,\xi)\in\bR^{2d}
$$
in the sense that
$$
a(x,\xi)\|\phi\|_\fH^2+\la\phi|B|\phi\ra\le\la\phi|c(x,\xi)|\phi\ra\text{ for each }\phi\in\fDom(c(x,\xi))\,. 
$$
Since $v_k(x,\xi)\in\fDom(c(x,\xi))$ for a.e. $(x,\xi)\in\bR^{2d}$ and each $k\ge 1$
$$
\ba
a(x,\xi)p(x,\xi)+\Tr_\fH(Q(x,\xi)B)
\\
=
a(x,\xi)\Tr_\fH(Q(x,\xi))+\Tr_\fH(Q(x,\xi)^{1/2}BQ(x,\xi)^{1/2})
\\
=a(x,\xi)\sum_{k\ge 1}\la v_k(x,\xi)|v_k(x,\xi)\ra+\sum_{k\ge 1}\la v_k(x,\xi)|B|v_k(x,\xi)\ra
\\
\le\sum_{k\ge 1}\la v_k(x,\xi)|c(x,\xi)|v_k(x,\xi)\ra=\Tr_\fH(Q(x,\xi)^{1/2}c(x,\xi)Q(x,\xi)^{1/2})&\,.
\ea
$$
Integrating in $(x,\xi)$ shows that
$$
\ba
\int_{\bR^{2d}}a(x,\xi)p(x,\xi)dxd\xi+\Tr_\fH(SB)
\\
\le\int_{\bR^{2d}}\Tr_\fH(Q(x,\xi)^{1/2}c(x,\xi)Q(x,\xi)^{1/2})dxd\xi
\ea
$$
since, by construction,
$$
\int_{\bR^{2d}}Q(x,\xi)dxd\xi=S\,.
$$
Thus
$$
\ba
\la\L,c\ra=\sup_{a\in C_b(\bR^{2d}),\,B=B^*\in\cL(\fH)\atop a(x,\xi)I_\fH+B\le c(x,\xi)}\left(\int_{\bR^{2d}}a(x,\xi)p(x,\xi)dxd\xi+\Tr_\fH(SB)\right)
\\
\le\int_{\bR^{2d}}\Tr_\fH(Q(x,\xi)^{1/2}c(x,\xi)Q(x,\xi)^{1/2})dxd\xi\le\la\L,c\ra&\,,
\ea
$$
where the first equality follows from convex duality as explained in Step 2, while the last inequality has been obtained above at the beginning of Step 6. This completes the proof.

\smallskip
It remains to prove that
$$
\int_{\bR^{2d}}\Tr_\fH(Q(x,\xi)^{1/2}c(x,\xi)Q(x,\xi)^{1/2})dxd\xi<\infty\,.
$$
Since
$$
c(x,\xi)\le(|x|^2+|\xi|^2)I_\fH+H
$$
one has
$$
v_k(x,\xi)\in\fDom(H)\implies v_k(x,\xi)\in\fDom(c(x,\xi))
$$
and
$$
\la v_k(x,\xi)|c(x,\xi)|v_k(x,\xi)\ra\le(|x|^2+|\xi|^2)\|v_k(x,\xi)\|_\fH^2+\la v_k(x,\xi)|H|v_k(x,\xi)\ra\,.
$$
Let $(h_1,h_2,\ldots)$ be a complete orthonormal system of eigenvectors of $H$ in $\fH$ (the Hermite functions), with eigenvalues $\mu_j$. Since
$$
\sum_{k\ge 1}\overline{t_{km}(x,\xi)}t_{kn}(x,\xi)=\la e_m(x,\xi)|Q(x,\xi)|e_n(x,\xi)\ra
$$
by definition of $t_{kl}(x,\xi)$, one has
$$
\ba
\sum_{k\ge 1}|\la v_k(x,\xi)|h_j\ra|^2=\sum_{k\ge 1}\sum_{m,n\ge 1}\overline{t_{km}(x,\xi)}t_{kn}(x,\xi)\la e_m(x,\xi)|h_j\ra\la h_j|e_n(x,\xi)\ra
\\
=\sum_{m,n\ge 1}\la e_m(x,\xi)|Q(x,\xi)|e_n(x,\xi)\ra\la e_m(x,\xi)|h_j\ra\la h_j|e_n(x,\xi)\ra=\la h_j|Q(x,\xi)|h_j\ra&\,.
\ea
$$
Hence
$$
\ba
\int_{\bR^{2d}}\sum_{k\ge 1}\la v_k(x,\xi)|H|v_k(x,\xi)\ra dxd\xi=\sum_{j\ge 1}\mu_j\int_{\bR^{2d}}\la h_j|Q(x,\xi)|h_j\ra
\\
=\sum_{j\ge 1}\mu_j\la h_j|S|h_j\ra=\Tr_\fH(S^{1/2}|H|S^{1/2})<\infty&\,,
\ea
$$
and since
$$
\sum_{k\ge 1}\|v_k(x,\xi)\|_\fH^2=\Tr_\fH(Q(x,\xi))=p(x,\xi)\,,
$$
one concludes that
$$
\ba
\int_{\bR^{2d}}\Tr_\fH(Q(x,\xi)^{1/2}c(x,\xi)Q(x,\xi)^{1/2})dxd\xi
\\
\le\int_{\bR^{2d}}(|x|^2+|\xi|^2)p(x,\xi)dxd\xi+\Tr_\fH(S^{1/2}HS^{1/2})<\infty&\,.
\ea
$$

\end{proof}
%\subsection{Optimality}\label{optimality}
%%%%%%%%%%%%%%%%%%%%%%%%%%%%%%%%%%%%%%%%%%%%%%%%%%%%%%%%%%%%%%%%%%%%%%%%%%%%%%%%%%%%%%%%%%%%%%%%%%%%%%%%%%%%%%%%%%%%%%%
\section{Applications of duality for $\cE_\hb$ I: inequalities between $
MK_\hb$, $\cE_\hb$ and $\MKd$. }\label{app1}
%%%%%%%%%%%%%%%%%%%%%%%%%%%%%%%%%%%%%%%%%%%%%%%%%%%%%%%%%%%%%%%%%%%%%%%%%%%%%%%%%%%%%%%%%%%%%%%%%%%%%%%%%%%%%%%%%%%%%%%

\begin{Thm}\lb{T-LBdMKE}
Let $R,S\in\cD_2(\fH)$ and $p$ be a probability density on $\bR^{2d}$. Then
\begin{eqnarray}
\cE_\hb(\tilde W_\hb(R),S)^2&\geq& \MKd(\widetilde W_\hb[R],\widetilde W_\hb[S]- d\hb,
\nonumber\\
MK_\hb(R,S)^2&\ge&\cE_\hb(\tilde W_\hb(R),S)^2-d\hb\,,
\nonumber\\
MK_\hb(R,S)^2&\ge& \MKd(\widetilde W_\hb[R],\widetilde W_\hb[S]-2d\hb\,.\nonumber
\end{eqnarray}
\end{Thm}

\begin{proof}
The first inequality and the third inequality (also a consequence of the two others) were proved in  Theorem 2.4 (2) of \cite{FGPaul} and Theorem 2.3 (2) of \cite{FGMouPaul} respectively.

The second inequality is proved along the same lines as Theorem 2.3 (2) of \cite{FGMouPaul}. Let $a\equiv a(x,\xi)$ in $C_b(\bR^{2d};\bR)$ and $B=B^*\in\cL(\fH)$ satisfy
$$
a(x,\xi)I_\fH+B\le c(x,\xi)\qquad\text{ for a.e. }(x,\xi)\in\bR^{2d}\,.
$$
Then
$$
a(x,\xi)|x,\xi\ra\la x,\xi|\otimes I_\fH+|x,\xi\ra\la x,\xi|\otimes B\le |x,\xi\ra\la x,\xi|\otimes c(x,\xi)
$$
for a.e. $(x,\xi)\in\bR^{2d}$, so that
$$
\ba
\Op^T_\hb((2\pi\hb)^da)\otimes I_\fH+I_\fH\otimes B\le&\tfrac1{(2\pi\hb)^d}\int_{\bR^{2d}}|x,\xi\ra\la x,\xi|\otimes c(x,\xi)dxd\xi
\\
=&C+d\hb I_{\fH\otimes\fH}\,.
\ea
$$
Thus, for each $\cQ\in\cC(R,S)$, one has
$$
\ba
\Tr_{\fH\otimes\fH}(\cQ^{1/2}C\cQ^{1/2})+d\hb\ge&\Tr_{\fH\otimes\fH}(\cQ^{1/2}(\Op^T_\hb((2\pi\hb)^da)\otimes I_\fH+I_\fH\otimes B)\cQ^{1/2})
\\
=&\Tr_{\fH\otimes\fH}(\cQ(\Op^T_\hb((2\pi\hb)^da)\otimes I_\fH+I_\fH\otimes B))
\\
=&\Tr_{\fH}(R\Op^T_\hb((2\pi\hb)^da)+SB)
\\
=&\int_{\bR^{2d}}a(x,\xi)\tilde W_\hb(R)(x,\xi)dxd\xi+\Tr_\fH(SB)\,.
\ea
$$
In particular
$$
\ba
MK_\hb(R,S)^2+d\hb=\inf_{\cQ\in\cC(R,S)}\Tr_{\fH\otimes\fH}(\cQ^{1/2}C\cQ^{1/2})+d\hb
\\
\ge\sup_{a\in C_b(\bR^{2d},\bR),\,B=B^*\in\cL(\fH)\atop a(x,\xi)I_\fH+B\le c(x,\xi)}\left(\int_{\bR^{2d}}a(x,\xi)\tilde W_\hb(R)(x,\xi)dxd\xi+\Tr_\fH(SB)\right)
\\
=\cE_\hb(\tilde W_\hb(R),S)^2&\,.
\ea
$$
\end{proof}

%%%%%%%%%%%%%%%%%%%%%%%%%%%%%%%%%%%%%%%%%%%%%%%%%%%%%%%%%%%%%%%%%%%%%%%%%%%%%%%%%%%%%%%%%%%%%%%%%%%%%%%%%%%%%%%%%%%%%%%
\section{Applications of duality for $\cE_\hb$ II:  ``triangle'' inequalities}\label{app2}
%%%%%%%%%%%%%%%%%%%%%%%%%%%%%%%%%%%%%%%%%%%%%%%%%%%%%%%%%%%%%%%%%%%%%%%%%%%%%%%%%%%%%%%%%%%%%%%%%%%%%%%%%%%%%%%%%%%%%%%

\begin{Thm}\lb{T-IneqT2}
Let $R,S,T\in\cD_2(\fH)$ and let $f,g\in\cP_2(\bR^{2d})$. Then

\noindent
(i) one has
$$
\ba
\MKd(f,g)\le&\sqrt{\cE_\hb(f,S)^2+d\hb}+\sqrt{\cE_\hb(g,S)^2+d\hb}
\\
<&\cE_\hb(f,S)+\cE_\hb(g,S)+d\hb\,;
\ea
$$
(ii) one has
$$
\ba
\cE_\hb(f,T)\le&\MKd(f,\tilde W_\hb(S))+\cE_\hb(\tilde W_\hb(S),T)
\\
\le&\sqrt{\cE_\hb(f,S)^2+d\hb}+\sqrt{MK_\hb(S,T)^2+d\hb}
\\
<&\cE_\hb(f,S)+MK_\hb(S,T)+d\hb\,;
\ea
$$
(iii) one has
$$
\ba
MK_\hb(R,T)\le&\cE_\hb(\tilde W_\hb(S),R)+\cE_\hb(\tilde W_\hb(S),T)
\\
\le&\sqrt{MK_\hb(R,S)^2+d\hb}+\sqrt{MK_\hb(S,T)^2+d\hb}
\\
<&MK_\hb(R,S)+MK_\hb(S,T)+d\hb\,.
\ea
$$
\end{Thm}

\begin{proof}
The triangle inequality for $\MKd$ implies that
$$
\MKd(f,g)\le\MKd(f,\tilde W_\hb(S))+\MKd(\tilde W_\hb(S),g)\,.
$$
Then, Theorem 2.4 (2) of \cite{FGPaul} implies that
$$
\ba
\MKd(f,\tilde W_\hb(S))\le\sqrt{\cE_\hb(f,S)+d\hb}\,,
\\
\MKd(\tilde W_\hb(S),g)\le\sqrt{\cE_\hb(g,S)+d\hb}\,.
\ea
$$
This implies the first inequality in (i). As for the second inequality, for each $X,Y>0$, one has the obvious elementary inequality
$$
\sqrt{X+Y}<X+\tfrac12Y\,.
$$
This inequality obviously applies to the present case since $\cE_\hb(f,S)\ge d\hb$ and $\cE_\hb(g,S)\ge d\hb$ by Theorem 2.4 (2) of \cite{FGPaul}. This proves (i).

Observe that the first inequality in (ii) is inequality (a) in Theorem \ref{T-IneqT} with $g=\tilde W_\hb(S)$ and $R_1=T$. Then Theorem 2.4 (2) of \cite{FGPaul} implies that
$$
\MKd(f,\tilde W_\hb(S))\le\sqrt{\cE_\hb(f,S)^2+ d\hb}\,,
$$
while Theorem \ref{T-LBdMKE} implies that
$$
\cE_\hb(\tilde W_\hb(S),T)\le\sqrt{MK_\hb(S,T)^2+ d\hb}\,,
$$
and this implies the second inequality in (ii). The third inequality is obtained as in (i).

Finally, the first inequality in (iii) is inequality (b) in Theorem \ref{T-IneqT} with $R_1=R$, while $R_3=T$ and $f=\tilde W_\hb(S)$. Then, Theorem \ref{T-LBdMKE} implies that
$$
\ba
\cE_\hb(\tilde W_\hb(S),R)\le\sqrt{MK_\hb(R,ST)^2+ d\hb}\,,
\\
\cE_\hb(\tilde W_\hb(S),T)\le\sqrt{MK_\hb(S,T)^2+ d\hb}\,,
\ea
$$
which gives the second inequality in (iii). Finally, the third inequality is obtained as in (i).
\end{proof}

\smallskip
\noindent
\textbf{Remark.} It is interesting to compare the inequality (iii) above with the ``genereliazed triangle inequality'' in \cite{DePalma}. Let us recall that DePalma and Trevisan have constructed a pseudo-distance on density operators on $\fH$ 
which is similar to ours to some extent. The DePalma-Trevisan distance $D$ is defined through a different notion of coupling than in \cite{FGMouPaul}; specifically, their notion of couplings is based on ``quantum channels'' (completely positive
linear maps on the set of density operators): see Definition 1 in \cite{DePalma}. While the transport cost in formula (19) of \cite{DePalma} is in some sense reminiscent of the transport cost used in \cite{FGMouPaul}, these two costs are in fact
significantly different. For instance, the transport cost used in the definition of $MK_\hb$ in \cite{FGMouPaul}, and in the present paper, has compact resolvent, and therefore its spectrum consists of eigenvalues only. On the contrary, the cost 
operator in \cite{DePalma} in the case of Gaussian quantum systems has continuous spectrum on $[0,+\infty)$.

In Theorem 2 of \cite{DePalma}, DePalma and Trevisan prove what they call a ``triangle inequality'' for their distance $D$, of the form
$$
D(R,T)\le D(R,S)+D(S,S)+D(S,T)
$$
(inequality (35) in \cite{DePalma}). Of course, if $D$ was a real distance, $D(S,S)=0$, and the inequality above coincides with the usual triangle inequality. In \cite{DePalma}, there is an explicit formula for $D(S,S)$ in terms of the canonical 
purification of $S$ (Corollary 1, formula (34) in \cite{DePalma}).

With the distance $MK_\hb$ defined in \cite{FGMouPaul}, one has
$$
MK_\hb(R,S)\ge 2d\hb\,,\qquad\text{ for all }R,S\in\cD_2(\fH)\,,
$$
so that Theorem \ref{T-IneqT2} (iii) implies that
$$
MK_\hb(R,T)<MK_\hb(R,S)+MK_\hb(S,S)+MK_\hb(S,T)\,.
$$
In other words, $MK_\hb$ satisfies the same ``generalized triangle inequality'' as the DePalma-Trevisan distance $D$, with a strict inequality.

\section{Applications of duality for $\cE_\hb$ III:  Classical/quantum optimal transport and   semiquantum Legendre transform}\label{app3}
\subsection{A classical/quantum optimal transport}\label{opckaqua}
%Soit une densit\'e de proba $r$ sur $\mathbb R^{2d}$ et une matrice densit\'e $S$ sur $L^2(\mathbb R^d)$. 
% 
% On d\'efinit $E(r,S)$ comme la ``distance" ARMA.
 
%On se place dans la situation o\`u la dualit\'e de Kanto marche et l'on a donc une fonction 
 Let $r$ be a probability density on $\bR^{2d}$ and $S$ a density operator on $L^2(\bR^d)$.

We suppose that an optimal operator $\widetilde B$ and an optimal function $\widetilde a$ exists for the Kantorovich duality formulation of  $\cE_\hb(r,S)$, as  in Theorem \ref{T-DualityARMA}, and  that $\widetilde a\in C_b(\bR^{2d})$ and $\widetilde B\in\cL(\fH)$. That is to say that 
%there exists a function $a'(z),\ z\in \mathbb R^{2d}$ and an operator $B'$, such that
$$
\widetilde a(q,p)+\widetilde B\leq (Z-z)^2\mbox{ and }\cE_\hb(r,S)^2=\int_{\bR^{2d}}\widetilde a(z)r(z)dz+\Tr_{L^2(\bR^d)}{(\widetilde BS)}.
$$
Here we have used the notation $z=(q,p)$, $dz=dqdp$, and $Z=(Q,P)$. 

Let us denote by $\Pi(z)$ an optimal coupling of $r,S$ and let us define \begin{eqnarray}
a(z)&:=&\tfrac12(|z|^2-\widetilde a(z))\nn\\
B&:=&\tfrac12(|Z|^2-\widetilde B).\nn
\end{eqnarray}
 One has
$$
(a(z)+B-z\cdot Z)\geq 0 \mbox { and }\Tr_{L^2(\bR^d)}\int_{\bR^{2d}}\Pi(z)^{\frac12}( a(z)+B-z\cdot Z)\Pi(z)^{\frac12}dz=0,
$$
Therefore, since $\Pi(z)^{\frac12}( a(z)+B-z\cdot Z)\Pi(z)^{\frac12}\geq 0$
$$
\Pi(z)^{\frac12}( a(z)+B-z\cdot Z)\Pi(z)^{\frac12}=0\mbox{ a.e},.
$$ 
In other words,
$$
\Pi(z)^{\frac12}( a(z)+B-z\cdot Z)^{\frac12}
\left(\Pi(z)^{\frac12}( a(z)+B-z\cdot Z)^{\frac12}\right)^*
%( a(z)+B-2z\cdot Z)^{\frac12}\Pi(z)^{\frac12}
=0
$$
which implies that
%\ \Rightarrow 
$$( a(z)+B-z\cdot Z)^{\frac12}\Pi(z)^{\frac12}=0 \mbox { a.e.}
$$
and (forgetting the ``a.e." in the sequel)
\be\label{ae}
( a(z)+B-z\cdot Z)\Pi(z)=0.
\ee
% $\pi:\mathcal P(\mathbb R^{2d})\to\mathcal D(L^2(\mathbb R^d))$ est un couplage optimal de $r$ et $S$. 

Hence, the range of  $\Pi$ consists in functions $\mathbb R^{2d}\ni z\mapsto\psi_z\in L^2(\mathbb R^d)$ such that

\be\label{eqvp}
( a(z)+B-z\cdot Z)\psi_z=0\Longleftrightarrow
(B-z\cdot Z)\psi_z=-a(z)\psi_z:
\ee
\centerline{\textit{the vectors $\psi_z$ are the eigenvectors of   $B-z\cdot Z$ with eigenvalue $-a(z)$.}}

 But 
 %we have that 
 $B+a(z)-z\cdot Z\geq 0$. Therefore 
 
 \centerline{\textit{$-a(z)$ is the lowest eigenvalue of $B-2z\cdot Z$.}} 
%(supposons que le spectre de $(z\cdot Z--B)$ n'est pas d\'eg\'en\'er\'e pour l'instant).
 
 From now on, we will suppose that the fundamental of $B-z\cdot Z$ is non degenerate. This means that $\Pi(z)$ is proportional to 
 $|\psi_z\rangle\langle\psi_z|$ and therefore, since $\Pi(z)$ is a coupling between $r$ and $S$, 
$$
\Pi(z)=r(z)|\psi_z\rangle\langle\psi_z|
$$
and
$$
S=\int_{\bR^{2d}}r(z)|\psi_z\rangle\langle\psi_z|dz.
$$

We just prove the following result.
\begin{Thm}\label{trans}
Let $B$ be a bounded optimal Kantorovich operator of $\cE_\hb(r,S)$. 
Let moreover, for each $z\in\bR^{2d}$, $\psi_z$ be the ground state of $B-z\cdot Z$.

Then $S$ admits the following representation
$$
S=\int_{\bR^{2d}}r(z)|\psi_z\rangle\langle\psi_z|dz.
$$
\end{Thm}
\vskip 1cm
Theorem \ref{trans} suggests
%, by analogy with the T\"oplitz quantization, 
to associate to any probability density $\mu$ the following operator
\be\label{deltopp}
\mu\longrightarrow
\Op^{r,S}_\hb[\mu]:=\int_{\bR^{2d}}|\psi_z\rangle\langle\psi_z|\mu(dz).
\ee

The arrow in \eqref{deltopp} can be seen as  the ``optimal transport", from classical probability densities to quantum density matrices,  transporting $r$ to $S$. 

Note that, for any  density $\mu$,
$$
\Tr\Op^{r,S}_\hb[\mu]=\int_{\bR^{2d}}\mu(dz).
$$

Finally, using \eqref{ae}, we easily show, by analogy with the proof of Theorem 2.6 (b) in \cite{ECFGTPaul}, that, 
%under some extra hypothesis of $a$ and $B$,
when $a\in C^1(\bR^{2d}),\ (\nabla a)r\in C_b(\bR^{2d})$ and, e.g.,  $\psi_z\in \mbox{Dom}(\tfrac1{i\hbar}[Z,B])$ for all $z\in supp(r)$,
\begin{eqnarray}
0&=&
\Pi(z)\tfrac1{i\hbar}[Z,( a(z)+B-z\cdot Z)\Pi(z)]=\Pi_z\tfrac1{i\hbar}[Z, a(z)+B-z\cdot Z]\Pi(z)\nn\\
&=&
\Pi(z)([Z, B]-z)\Pi(z)\nn\\
\mbox{and}&&\nn\\
0&=&\Pi(z)|{Z,( a(z)+B-z\cdot Z)\Pi(z)\}=\Pi(z)\{Z, a(z)+B-z\cdot Z|}\Pi(z)\nn\\
&=&
\Pi(z)(\nabla a(z)-Z)\Pi(z).\nn
\end{eqnarray}
Therefore the (classical and quantum) ``gradient" aspect appears in the following expressions
\begin{eqnarray}
\langle\psi_z|Z|\psi_z\rangle&=&\nabla a(z)\nn\\
z&=&\langle\psi_z|Z|\nabla^QB\psi_z\rangle\nn
\end{eqnarray}
where $\nabla^Q:=\tfrac1{i\hbar}[JZ,\cdot]$ with $J$ the symplectic matrix defined by $\{f,g\}=\nabla f\cdot J\nabla g$, as introduced and motivated in \cite[Section 1]{ECFGTPaul}.

\vskip 1cm
Let us finish this section by an example. Suppose that
$$
S=\Op^T_\hb((2\pi\hb)^dr).
$$
In this case, one knows, \cite[Theorem 2.4 (1)]{FGPaul} (note a difference of normalization: in \cite{FGPaul}, $E_\hb^2=\tfrac12\cE_\hb^2$),
$$
\cE_\hb(r,S)^2=d\hbar=\int_{\bR^{2d}}\widetilde a(z)r(z)dz+\Tr(\widetilde BS)\mbox{ with }\widetilde a=0,\ \widetilde B=d\hbar I_\fH .
$$
Since $(q-x)^2+(p+i\hbar\nabla_x)^2\geq d\hbar I_\fH=\widetilde a(z)I_\fH+\widetilde B$, $\widetilde a$ and $ \widetilde B$ are optimal and
$$
a(q,p)=\tfrac12{|z|^2}\mbox{ and } B=\tfrac12(|Z|^2-d\hbar).
$$
Hence
$$
a(z)+B-z\cdot Z=\tfrac12(-\nabla_x+x-(q+ip))(\nabla_x+x-(q-ip)),
$$
 the solution of \eqref{eqvp} is
$$
\psi_z=(\pi\hbar)^{-d/4}e^{-\frac{(x-q)^2}{2\hbar}}e^{i\frac{p.x}\hb}
$$
and Theorem \ref{trans} expresses back that $
S=\Op^T_\hb((2\pi\hb)^dr)
$ and 
$$\Op^{\mu,\Op^T_\hb(\mu)}_\hb=\Op^T_\hb
$$
 for any probability density $\mu$.

%Let us suppose now that
%$$
%S=\Op^T_\hb((2\pi\hb)^ds),\ s\neq r.
%$$

%\vskip 3cm
%De plus, par la caract\'esisation spectrale des valeurs propres
%$$
%a(z)=\inf_{\phi\in 
%%L^2(\mathbb R^d)
%\Dom(B)}(z.\langle\phi|Z|\phi\rangle-
%\langle\phi|B|\phi\rangle).
%$$
%Je pense que l'infimum peut \^etre un extremum, mais on a l\`a la version semiquantique de la transform\'ee de Legendre ou Lax-Hopf.

%Moreover, by the same argument than in the proof of Theorem 2.6 in \cite{ECFGTPaul}, we get easily that, 
%%under the condition, e.g., that $B$ is bounded,
%
%%Mais il y a plus. En appliquant l' argument de la d\'erivation par commutateur, on obtient que, sous la condition de simplicit\'e du spectre de $(z\cdot Z--B)$, 
%%%et en prenant une base orthonorm\'ee de $L^2(\mathbb R^d)$ de la forme $\{\phi_0=\psi_z,\phi_1,\dots\}$, 
%%on obtient que  (je prends la constante de Planck \'egale \`a un)
%$$ 
%\langle\psi_z|(z-\tfrac1{i\hbar}[JZ,B])\psi_z\rangle=0.
%%\ \forall i=0,\dots.
%$$
%Therefore
%%$$
%%\langle\psi_z|(z-\frac1{i\hbar}[JZ,B])\psi_z\rangle=0,
%%$$
%%soit 
%$$
%\langle\psi_z|\tfrac1{i\hbar}[JZ,B]\psi_z\rangle=z\
%$$
%
% and appears the ``transport" $z\to \tfrac1{i\hbar}[JZ,B]$ in the transformation
% $$
% r\to S.
% $$
 
\subsection{A semiquantum Legendre
%-Lax-Hopf 
transform}\label{llht}
  As we have seen, $-a(z)$ is the fundamental of the operator $B-z\cdot Z$. Therefore, by the variational characterization of the lowest eigenvalue, 
 %De plus, par la caract\'esisation spectrale des valeurs propres
$$
-a(z)=\inf_{\substack{\phi\in 
%L^2(\mathbb R^d)
\Dom(B)\\\|\phi\|_\fH=1}}(
\langle\phi|B|\phi\rangle-z\cdot\langle\phi|Z|\phi\rangle),
$$
to be faced to the classical definition of the Legendre transform
$$
a(z)=\sup_{z'}(z\cdot z'-
b(z')).
$$
Let us define the semiquantum Legendre transform by
$$
B^{sq*}:=\sup_{\substack{\phi\in 
%L^2(\mathbb R^d)
\Dom(B)\\\|\phi\|_\fH=1}}(z\cdot\langle\phi|Z|\phi\rangle-
\langle\phi|B|\phi\rangle).
$$
\begin{Thm}\label{sqleg}
Let $a(z)=\tfrac12(|z|^2-\widetilde a(z)),
B=\tfrac12(|Z|^2-\widetilde B)$ where $\widetilde a(z)$ and $
\widetilde B$ are  bounded optimal Kantorovich potentials for $\cE_\hb(r,S)$. Then
$$
a=B^{sq*}.
$$
\end{Thm}
\begin{proof}
We just recall the variational argument. 

Let $A\geq 0$  and $A|\phi_0\rangle=0$.  Then, 
$$
\langle\phi_0|A|\phi_0\rangle\leq
\inf_{\substack{\phi\in 
%L^2(\mathbb R^d)
\Dom(B)\\\|\phi\|_\fH=1}}
\langle\phi|A|\phi\rangle
$$
 and
 %  by \eqref{deltaaa},
$$
\langle\phi_0+\delta\phi_0|A|\phi_0+\delta\phi_0\rangle=\langle\delta\phi|A|\delta\phi).
$$
%Moreover, 
\end{proof}
\vskip 5cm

\end{document}